\documentclass[preprint,11pt]{imsart}

\usepackage{amsmath}
\usepackage{amsfonts}
\usepackage{mathrsfs}

\usepackage{ amsmath, amsfonts, amssymb, amsthm, amscd}
\usepackage[T1]{fontenc}
\usepackage[english]{babel}
\usepackage{color}
\usepackage{pgf,pgfarrows,pgfnodes,pgfautomata,pgfheaps}
\usepackage{amsmath,amssymb}
\usepackage[latin1]{inputenc}

\usepackage[latin1]{inputenc}
\usepackage{graphicx}

\fontsize{11pt}{14.5pt}\selectfont

\def\Dy#1{\Frac{\partial #1}{\partial y}}

\def\Dy_1y_1#1{\Frac{\partial^2 #1}{\partial y_1^2}}

\newtheorem{Theorem}{Theorem}[part]
\newtheorem{Definition}{Definition}[part]
\newtheorem{Proposition}{Proposition}[part]
\newtheorem{Assumption}{Assumption}[part]
\newtheorem{Lemma}{Lemma}[part]

\newtheorem{Remark}{Remark}[part]

\makeatletter \@addtoreset{equation}{section}

\@addtoreset{Definition}{section}

\@addtoreset{Theorem}{section}

\@addtoreset{Proposition}{section}

\@addtoreset{Assumption}{section}

\@addtoreset{Corollary}{section}

\@addtoreset{Lemma}{section}

\@addtoreset{Remark}{section}

\@addtoreset{Example}{section}

\makeatother \makeatletter

\def \Int{\displaystyle\int}
\def \Frac{\displaystyle\frac}

\def \Sup{\displaystyle\sup}

\def \be{\begin{eqnarray}}
\def \ee{\end{eqnarray}}
\def \b*{\begin{eqnarray*}}
\def \e*{\end{eqnarray*}}

\def \E{\mathbb{E}}

\def \L{\mathbb{L}}
\def \N{\mathbb{N}}
\def \P{\mathbb{P}}
\def \R{\mathbb{R}}

\def \W{\overleftarrow{W}}
\def \B{\overleftarrow{B}}

\def \[{[\,\!\![}
\def \]{]\,\!\!]}

\def \1{{\bf 1}}

\def \ep{\hbox{ }\hfill$\Box$}

\def\Fc{{\cal F}}

\def\Hc{{\cal H}}

\def\Lc{{\cal L}}

\def\Nc{{\cal N}}

\def\Lc{{\cal L}}

\def\Sc{{\cal S}}

\begin{document}

\begin{frontmatter}

\title{ Backward Doubly SDEs and  Semilinear Stochastic PDEs in a convex domain}
\runtitle{BDSDEs in a convex domain}


\author{\fnms{Anis} \snm{Matoussi}\ead[label=e1]{anis.matoussi@univ-lemans.fr}\thanksref{t2}}
\thankstext{t2}{The first author
 was  partially supported by  chaire Risques Financiers de la fondation du risque, CMAP-\'Ecole Polytechniques, Palaiseau-France  }
\address{University of Maine \\
Risk and Insurance Institute of Le Mans \\
Laboratoire  Manceau de Math\'ematiques\\ Avenue Olivier Messiaen\\
 \printead{e1}}
  \author{\fnms{Wissal} \snm{Sabbagh}\corref{}\ead[label=e2]{wissal.sabbagh@univ-lemans.fr}}
\address{ University of Maine \\
Risk and Insurance Institute of Le Mans \\
Laboratoire Manceau de Math\'ematiques\\ Avenue Olivier Messiaen \\ \printead{e2}}
 \affiliation{University of Le Mans}
 
\author{\fnms{Tusheng} \snm{Zhang}\corref{}\ead[label=e3]{tusheng.zhang@manchester.ac.uk}}
\address{ School of Mathematics, University of Manchester, \\
Oxford Road, Manchester M13 9PL, England, UK \\
 \printead{e3}}
 \affiliation{University of Manchester}


\runauthor{Matoussi, Sabbagh and Zhang }

\begin{abstract} 
This paper presents existence and uniqueness results for reflected  backward doubly stochastic differential equations (in short RBDSDEs) in a convex domain  $D$ without any regularity conditions on the boundary. Moreover,  using a stochastic flow approach a probabilistic interpretation for a system  of reflected SPDEs in a domain is given via such RBDSDEs. The solution is expressed as a pair  $(u,\nu) $ where 
$u$ is a predictable continuous process which takes values in a  Sobolev space and $\nu$ is a random regular measure. The bounded variation process $K $, the component of the  solution of the reflected BDSDE, controls the set when   $ u$ reaches the boundary of $D$. This bounded variation process determines the measure $\nu$ from a particular relation by using the inverse of the flow associated to the the diffusion operator.

\end{abstract}


\end{frontmatter}

\section{Introduction}
Our main interest is  the following  system of semilinear stochastic PDE with value in  $\R^k$,
\small
\begin{equation}
\begin{split}
\label{SPDE1}   du_t (x) + [\mathcal {L}u_t (x) +f_t(x,u_t (x),\nabla u_t\sigma (x))] dt +  h_t(x,u_t(x),\nabla u_t\sigma(x))\cdot d\W_t  = 0, \,
\end{split}
\end{equation}
over the time interval $[0,T]$. The final condition is given by $u_T = \Phi$, $f,$
$h$ are  non-linear random functions  and  $\mathcal L$ is the second order differential operator  which is defined
 component by component by
 \begin{equation}\label{operator conv}
 \begin{array}{lll}
 {\mathcal L}\varphi(x)&=&\displaystyle\sum_{i=1}^{d}b^i(x)\frac{\partial}{\partial
 x_i}\varphi(x)+\frac{1}{2}\sum_{i,j=1}^{d}a^{ij}(x)\frac{\partial^2}{\partial
 x_i\partial x_j}\varphi(x).
\end{array}
\end{equation}
The integral term with $d\W_t$
refers to the backward stochastic integral with respect to a $d$-dimensional Brownian motion on
$\big(\Omega, \mathcal{F},\mathbb{P}, (W_t)_{t\geq 0} \big)$. We use the backward notation  because in the proof we will employ the backwards doubly stochastic framework introduced by Pardoux and Peng \cite{pp1994}.\\
 Such SPDEs appear in various applications like pathwise stochastic control problems, the Zakai equations in filtering and stochastic control with partial observations.  It is well known now that  BSDEs give a probabilistic interpretation for the solution of a class of semi-linear PDEs. 
By introducing in  standard BSDEs a second nonlinear term driven by an external noise, we obtain Backward Doubly SDEs (BDSDEs in short) \cite{pp1994} (see also \cite{BM01}, \cite{MS02}), which can be seen as Feynman-Kac representation of SPDEs and provide a powerful tool for probabilistic  numerical  schemes \cite{matouetal13}  for such SPDEs. Several generalizations to investigate more general nonlinear SPDEs have been developed following different approaches of the notion of weak solutions, namely,  Sobolev's solutions  \cite{DS04, GR00, Krylov99, SSV03, Walsh},  and  stochastic viscosity  solutions \cite{lion:soug:98, lion:soug:00, lion:soug:01, buck:ma:10a, buck:ma:10b}. \\

 Given  a  convex domain $ D$ in $ \R^k$, our paper is concerned with the study of  weak solutions  to the reflection problem for  multidimensional SPDEs  \eqref{SPDE1}   in $D$ by introducing the associated BDSDE. \\ 
Inspired by the variational  formulation  of the obstacle problem for SPDEs and  Menaldi's work  \cite{M83}  on  reflected diffusion, we consider the  solution of the  ref{\color{blue}{l}}ection problem for the SPDEs \eqref{SPDE1}  as a  pair $ (u, \nu)$, where $ \nu$ is a random regular measure and $ u \in \mathbf{L}^2 \big(\Omega \times [0,T]; H^1 (\mathbb R^d)\big)$ satisfies the following relations :
\small
\begin{equation}
\begin{split}
\label{OSPDE1}
& (i) \; \;    u_t(x) \in \bar D , \quad d\mathbb{P}\otimes dt\otimes dx - \mbox{a.e.},  \\
& (ii)\;\; du_t (x) + \big[
 \;  \mathcal{L} u_t (x)  + f_t(x,u_t (x),\nabla u_t\sigma(x)) \, \big]\,  dt +  h_t(x,u_t(x),\nabla u_t\sigma(x))\cdot d\W_t  = - \nu (dt,dx), \quad a.s.,  \\
& (iii)\; \; \nu(u\notin \partial  D)=0 , a.s.,\\
 & (iv)  \; \;  u_T = \Phi, \quad dx-\mbox{a.e.}.
\end{split}
\end{equation}
$\nu$  is a random measure which  acts only when  the process $u$ reaches the boundary of  the domain $ D$.
 The rigorous  sense of the relation $(iii)$ will be based on the  probabilistic representation of the measure $\nu$  in terms of the bounded variation processes $K$, a component of the associated  solution of the reflected  BDSDE in the domain $D$. This  problem is  well known as a Skorohod  problem for SPDEs.\\  
In the case of diffusion processes  in a domain,  the reflection problem has been investigated by severals authors (see  \cite{S61}, \cite{W71},  \cite{ElkChM80},   \cite{LS84}). In the case of a convex domain this reflection problem was treated by Tanaka \cite{T79} and Menaldi \cite{M83} by using the variational inequality and the convexity properties of the domain.\\
In the one dimensional case, the reflection problem  for nonlinear PDEs (or SPDEs) has been studied by using different approaches.  
 The work of El Karoui et al \cite{Elk2} treats the obstacle problem for viscosity solution of  deterministic semilinear PDEs  within the framework of backward stochastic differential equations (BSDEs in short). 
This increasing process determines in fact the measure from the relation $(ii)$.  
Bally et al \cite{BCEF} (see also Matoussi and Xu \cite{MX08})  point out that the continuity of this process allows the classical notion of strong variational solution to be extended (see Theorem 2.2 of \cite{BensoussanLions78} p.238) and express the solution to the obstacle as a pair $(u, \nu)$ where $ \nu$ is supported by the set  $\{u=g\}$. \\
Matoussi and Stoica \cite{MS10} have proved an existence and uniqueness result for the obstacle problem of backward quasilinear stochastic PDE on the whole space $\R^d $ and driven by a finite dimensional Brownian motion. The method is based on the probabilistic interpretation of the solution by us- ing the backward doubly stochastic differential equation. They have also proved that the solution $u$ is a predictable continuous process which takes values in a proper Sobolev space and $\nu$ is a random regular measure satisfying the minimal Skohorod condition. In particular, they gave for the regular measure $\nu$ a probabilistic interpretation in terms of the continuous increasing process $K$ where $(Y,Z,K)$ is the solution of a reflected generalized BDSDE.

On the other hand, M.Pierre \cite{Pierre, Pierre80}  has studied parabolic PDEs with obstacles using the parabolic potential as a tool. He proved that the solution uniquely exists and is quasi-continuous with respect to the analytical capacity.  Moreover he gave a representation of the reflected measure $\nu$ in terms of the associated regular potential and the approach used  is based  on analytical quasi-sure analysis.  More recently,  Denis, Matoussi and  Zhang \cite{DMZ12} have extended this approach 
for the obstacle problem of quasilinear SPDEs when the obstacle is regular in some sense and controlled by the solution of a SPDE.\\

Nualart and Pardoux \cite{Nualart} have studied the obstacle problem for a nonlinear heat equation on the spatial interval $[0,1]$ with Dirichlet boundary conditions, driven by an additive space-time white noise. They proved the existence and uniqueness of the solution and their method relied heavily on the results for a deterministic variational inequality. Donati- Martin and Pardoux \cite{Donati-Pardoux} generalized the model of Nualart and Pardoux. The nonlinearity appears both in the drift and in the diffusion coefficients. They proved the existence of the solution by penalization method but they did not obtain the uniqueness result. And then in 2009, Xu and Zhang solved the problem of the uniqueness, see \cite{XZ09}.  We note also that Zhang established in \cite{Z11} the existence and uniqueness of solutions of system \eqref{OSPDE1} in the forward case when $x$ belongs to $[0,1]$. He approximated the system of SPDEs by a penalized system and used a number of a priori estimates to prove the convergence of the solutions.
However, in all their models,  they do not consider the case where the coefficients depend on $\nabla u$.\\

Our contributions in this paper are as following: first of all,  reflected BDSDEs in the convex domain $D$ are introduced and results of existence and uniqueness of  such RBDSDEs are established. Next, the existence and uniqueness results of the solution $(u, \nu)$ of the reflection  problem for \eqref{SPDE1} are given in Theorem \ref{existence:RSPDE}. Indeed, a probabilistic method based on reflected BDSDEs and stochastic flow technics   are investigated in our context (see e.g. \cite{BM01}, \cite{MS02}, \cite{K94a, K94b}  for these flow technics). 
The key element in \cite{BM01} is to use  
the inversion of stochastic flow which transforms the variational
formulation of the SPDEs to the associated BDSDEs. Thus it plays the same role as It\^o's formula in the case of the classical solution of SPDEs.  \\
We also mention the works \cite{CEK11}, \cite{HZ10} and  \cite{HT10} where they have studied  a  Reflected BSDEs  with oblique reflection in multi-dimensional case and their relations   to switching  problems.\\
  
  This paper is organized as following: in Section 2, first the basic assumptions and the definitions of the solutions
for  Reflected BDSDE in a convex domain are presented.  Then,  existence and  uniqueness of  solution of RBDSDE (Theorem \ref{existence:RBDSDE}) is given under only convexity assumption for the domain without any regularity on the boundary. This result  is  proved  by using  penalization approximation. Thanks to the convexity properties we prove several technical lemmas, in particular the  fundamental Lemma \ref{fundamental:lemma}. In Section 3, we study  semilinear SPDE's in a convex domain.  We first  provide useful results on stochastic flow associated with the forward SDEs, then in this
 setting  as in Bally and Matoussi \cite{BM01}, an equivalence norm result associated to the diffusion process  is given. The main result of this section  Theorem \ref{existence:RSPDE} is the existence and uniqueness results
   of the solution of reflected SPDEs in a convex domain. The proof of this result is  based on the probabilistic interpretation via the Reflected Forward-BDSDEs.  The uniqueness is a consequence of the variational formulation of the SPDEs  written with random test functions and the uniqueness of the solution of the Reflected  FBSDE.
    The existence of the solution is established by an approximation penalization  procedure, a priori estimates 
    and the equivalence norm results.  In the Appendix, technical lemmas for the existence of the solution of the Reflected BDSDEs and SPDEs in a convex domain are given.

\section{Backward Doubly Stochastic Differential Equations in a domain}
\subsection{Hypotheses and preliminaries}
\label{hypotheses}
The euclidean norm of a vector $x\in\R^k$ will be denoted by $|x|$, and for a $k\times k$ matrix $A$, we define $\|A\|=\sqrt{Tr AA^*}$. In what folllows let us fix a positive number $T>0$.\\
Let $(\Omega, \Fc,\P)$ be a probability product space, and let $\{W_s, 0\leq s\leq T\}$ and $\{B_s, 0\leq s\leq T\}$ be two mutually independent standard Brownian motion processes, with values respectively in $\R^d$ and in $\R^l$.
For each $t\in[0,T]$, we define
$$\Fc_t:=\Fc_t^B\vee\Fc_{t,T}^W\vee \Nc$$
where $\Fc_t^B=\sigma\{B_r, 0\leq r\leq t\}$, $\Fc_{t,T}^W=\sigma\{W_r-W_t, t\leq r\leq T\}$ and $\Nc$ the class of $\P$ null sets of $\Fc$.
Note that the collection $\{\Fc_t, t\in[0,T]\}$ is neither increasing nor decreasing, and it does not constitute a filtration.
\subsubsection{Convexity results}
\label{convexity:subsubsection}
Besides, we need to  recall properties related to the convexity of a nonempty domain $D$ in $\R^k$.  Let $\partial D$ denote the boundary of $D$ and  $\pi(x)$ the orthogonal  projection of $x\in\R^k$ on  the closure $\bar{D}$. We have the following properties:
\be
(x'-x)^*(x-\pi(x))\leq 0,  ~~ \forall x\in\R^{\color{red}{k}}, ~ \forall x'\in \bar{D}\label{prop1}
\ee
\be
(x'-x)^*(x-\pi(x))\leq (x'-\pi(x'))^*(x-\pi(x)),  ~~ \forall x, x'\in\R^k\label{prop2}
\ee
\be
\exists a \in D, \gamma > 0,\, \mbox{such that}\,\, (x-a)^*(x-\pi(x))\geq \gamma|x-\pi(x)|, ~~ \forall x\in\R^k. \label{prop3}
\ee
For $x\in\partial D$, we denote by $n(x)$ the set of outward normal unit vectors at the point $x$.\\

 To avoid technical complications, we will focus our study on domain $D$ which satisfy the following assumption:
 \begin{Assumption}\label{assumptiondomain}
 \begin{itemize}
 \item[(i)] $D$ is a regular domain (i.e. a convex domain with class $C^2$ boundary).\
 \item[(ii)] $0\in D$.
 \end{itemize}
  
 \end{Assumption}
 \begin{Remark}
 If not we can approximate our convex domain $D$ by regular convex domains
Indeed, we define a sequence of regular convex domains which approximate uniformly $D$. Indeed, the function $h(x)= d(x,D) :=\underset{y\in D}{\inf}|x-y|$ is convex and uniformly continuous in $\R^k$. If we denote $(g_\delta)_{0\leq \delta\leq\delta_0}$ the approximation identity with compact supports, then $h_\delta=g_\delta*h$ is a sequence of regular convex functions which tends uniformly to $h$ as $\delta\rightarrow 0$. For a fixed $\eta >0$, $\{x, h_\delta(x) <\eta\}$ are regular convex domains that converge uniformly in the Hausdroff metriic to $\{x, d(x,D) <\eta\}$ when $\delta$ tends to $0$.  Letting $\eta\rightarrow 0$, we conclude that for all $\varepsilon >0$ there exists a regular convex domain $D_\varepsilon$  such that
\be
\underset{x\in D}{\Sup}\, d(x,D_\varepsilon) <\varepsilon \quad \mbox{and} \quad \underset{x\in D_\varepsilon}{\Sup}\, d(x,D) <\varepsilon
\label{propappro}
\ee
\end{Remark}
One can find all these results in Menaldi \cite{M83}, page 737.
\subsubsection{Functional spaces and  assumptions}
Hereafter, let us define the spaces and the norms which will be needed for the formulation of the BDSDE in a domain.\\
\begin{description}
\item[-] $\mathbf{L}^p_k({\mathcal F}_T)$ the space of $k$-dimensional
${\mathcal F}_T$-measurable random variables $\xi$ such that
$$\begin{array}{ll}
          \|\xi\|_{L^p}^p:=\E(|\xi|^p)<+\infty;
  \end{array}$$

\item[-] ${\mathcal H}^2_{k\times d}([0,T])$ the set of (classes of $d\P\otimes dt$ a.e. equal) $k\times d$-dimensional jointly measurable processes 
such that $Z_t$ is $\Fc_t$-measurable for a.e. $t\in[0,T]$ and 
$$\begin{array}{ll}
          \|Z\|_{{\mathcal H}^2}^2:= \E[\Int_{0}^{T}|Z_t|^2dt]<+\infty;
  \end{array}$$

\item[-] ${\mathcal S}^2_k([0,T])$ the space of $\R^{k}$-valued processes $Y=(Y_t)_{t\leq T}$, with continuous paths such that $Y_t$ is $\Fc_t$-measurable and 
$$\|Y\|_{{\mathcal S}^2}^2:= \E[\,\underset{t\leq T}{\Sup}\,|Y_t|^2]<+\infty;$$

\item[-] ${\mathcal A}^2_k([0,T])$ the space of $\R^{k}$-valued processes $K=(K_t)_{t\leq T}$, with continuous and bounded variation paths such that $K_t$ is $\Fc_t$-measurable, $K_0=0$ and 
$$\|K\|_{{\mathcal A}^2}^2:= \E[\,\underset{t\leq T}{\Sup}\,|K_t|^2]<+\infty.$$
\end{description}
We next state our main assumptions on the terminal condition $\xi$ and the functions $f$ and $h$:
\begin{Assumption}\label{Ass1} 
 $\xi\in\mathbf{L}^2_k({\Fc}_T)$ and $\xi\in\bar{D}$ a.s.
 \end{Assumption}
 \begin{Assumption}\label{Ass2}
 $f:\Omega\times [0,T]\times\R^k\times\R^{k\times d}\rightarrow\R^k ~,~ h:\Omega\times [0,T]\times\R^k\times\R^{k\times d}\rightarrow\R^{k\times l}$ are two random functions verifying:
\begin{itemize}

\item[\rm{(i)}] For all $(y,z)\in\R^k\times\R^{k\times d}$, $f_t(\omega,y,z)$ and $h_t(\omega,y,z)$ are $\Fc_t$- measurable.

 \item [\rm{(ii)}]$\E\big[\Int_0^T |f_t(0,0)|^2dt\big] < +\infty$ \quad,\quad $\E\big[\Int_0^T\|h_t(0,0)\|^2dt\big] < +\infty.$
 \item [\rm{(iii)}] There exist constants $c>0$ and $0<\alpha<1$ such that for any $(\omega,t)\in\Omega\times[0,T]~;~
(y_1,z_1),(y_2,z_2)\in\R^k\times\R^{k\times d}$
\b*
|f_t(y_1,z_1)-f_t(y_2,z_2)|^2 &\leq & c\big(|y_1-y_2|^2+\|z_1-z_2\|^2\big)\\
\|h_t(y_1,z_1)-h_t(y_2,z_2)\|^2 &\leq & c|y_1-y_2|^2+\alpha \|z_1-z_2\|^2.
\e*
\end{itemize}
\end{Assumption}
We denote by $f_t^0:=f_t(\omega,0,0)$ and $h_t^0:=h_t(\omega,0,0)$.\\

\noindent We add the following  further assumption:
\begin{Assumption}\label{Ass3}
\begin{itemize}
\item[\rm{(i)}] $\xi\in \mathbf{L}^4_k({\Fc}_T).$
\item[\rm{(ii)}] There exist $c>0$ and $0\leq \beta < 1$ such that for all $(t,y,z)\in [0,T]\times\R^k\times\R^{k\times d}$
$$h_t \, h_t^* (y,z)\leq c (Id_{\R^k}+ yy^*)+ \beta  \, zz^*.$$
\item[\rm{(iii)}] $f$ and $h$ are uniformly bounded in $(y,z)$.
\end{itemize}
\end{Assumption}
\begin{Remark}
\begin{enumerate}
\item The Assumption \ref{Ass3} \rm{(i)} and \rm{(ii)} are needed to prove the uniform $L^4$-estimate for $(Y^n,Z^n)$ solution of BDSDE \eqref{BDSDEpen} (see estimate \eqref{estimateL4} in the Appendix \ref{Proof of Lemma}). This is crucial for our proof of the fundamental {\color{blue}{L}}emma \ref{fundamental:lemma}.
\item  The Assumption \ref{Ass3} \rm{(iii)} is only added for simplicity and it can be removed by standard technics of BSDEs. The natural condition instead of \rm{(iii)} is $f^0$ and $h^0$ in $\mathbf{L}^4(\Omega, \Fc ,\P).$
\end{enumerate}

\end{Remark}

\noindent Now we introduce the definition of the solution of BDSDEs in a domain.
\begin{Definition}\label{definition:RBSDE}
The triplet of processes $(Y_t,Z_t,K_t)_{\{0\leq t\leq T\}} $ is a solution of the backward doubly stochastic differential equation in a convex domain $D$, with
terminal condition $\xi$ and coefficients $f$ and $h$, if the following hold:
\begin{description}
 \item[(i)]  $Y\in{\Sc}^2_k([0,T]) ~,~ Z\in{\Hc}^2_{k\times d}([0,T])$ and $K \in {\mathcal A}^2_k([0,T])$,
 \item[(ii)] \begin{equation} 
\label{RBDSDE} Y_t  = \xi +\Int_t^T f_s(Y_s,Z_s)ds +\Int_t^T h_s(Y_s,Z_s)d\W_s -\Int_t^T Z_sdB_s +K_T-K_t ~,~ 0\leq t\leq T \,\,a.s.
\end{equation}
 \item[(iii)] $Y_t\in\bar{D}~ ,~ 0\leq t\leq T,~ a.s.$
 \item[(iv)] for any continuous progressively measurable process $ (z_t)_{0 \leq t \leq T}$ valued in $\bar{D}$, 
 \begin{equation}
 \label{skorohod1}
 \Int_0^T (Y_t-z_t)^* dK_t \leq 0, \;  a.s.
 \end{equation}
\end{description}

The triplet $(Y_t,Z_t,K_t)_{\{0\leq t\leq T\}}$ is called a solution of RBDSDE with data $(\xi,f,h)$.
\end{Definition}

\vspace{0.2cm}
\begin{Remark}  From Lemma 2.1 in \cite{GPP95}, the condition \eqref{skorohod1} implies that the bounded variation process $K$ acts only when $Y$ reaches the boundary of the convex  domain $D$ and the so-called Skorohod condition  is satisfied:
\begin{equation}
\label{skorohod2}
 \int_0^T \1_{ \{Y_t \in D\} } dK_t =0.
 \end{equation}
Moreover there exits an $ \mathcal F_t$-measurable process $ (\alpha_t)_{0 \leq t \leq T}$ valued in $ \mathbb R^k$ such that $$ K_t = \Int_0^t \alpha_s d\|K_s\|_{VT} \quad \mbox{and} \; 
 - \alpha_s \in n(Y_s).$$
\end{Remark}
\vspace{0.1cm}
\noindent In the following, $C$ will denote a positive constant which doesn't depend  on $n$ and can vary from line to line.
\subsection{Existence and uniqueness of the solution}
\label{existenceBDSDE:section}
In this section we establish existence  and uniqueness results for RBDSDE \eqref{RBDSDE}.

\begin{Theorem}\label{existence:RBDSDE}
Let the Asumptions \ref{assumptiondomain}, \ref{Ass1}, \ref{Ass2} and \ref{Ass3} hold. Then, the RBDSDE \eqref{RBDSDE} has a unique solution $(Y,Z,K) \in {\mathcal S}^2_{k}([0,T]) \times {\mathcal H}^2_{k\times d}([0,T]) \times {\mathcal A}^2_{k}([0,T])$.
\end{Theorem}
\begin{proof}\\
{\bf{a) Uniqueness}}: Let $(Y^1,Z^1,K^1)$ and $(Y^2,Z^2,K^2)$ be two solutions of  the RBDSDE  \eqref{RBDSDE}. Applying generalized It\^o'{\color{blue}{s}} formula (Lemma 1.3 in \cite{pp1994} p.213) yields
\be \label{itouni}
|Y_t^1-Y_t^2|^2&+&\Int_t^T \|Z_s^1-Z_s^2\|^2ds  =  2\Int_t^T (Y_s^1-Y_s^2)^*(f_s(Y_s^1,Z_s^1)-f_s(Y_s^2,Z_s^2))ds\nonumber\\
& + & 2\Int_t^T (Y_s^1-Y_s^2)^*(h_s(Y_s^1,Z_s^1)-h_s(Y_s^2,Z_s^2))d\W_s -2\Int_t^T (Y_s^1-Y_s^2)(Z_s^1-Z_s^2)dB_s\nonumber\\
&+& 2\Int_t^T (Y_s^1-Y_s^2)^*(dK_s^1-dK_s^2) + \Int_t^T \|h_s(Y_s^1,Z_s^1)-h_s(Y_s^2,Z_s^2)\|^2ds. 
\ee
Moreover, under the minimality condition (iv) we have 
\be\label{minest}
\Int_t^T (Y_s^1-Y_s^2)^*(dK_s^1-dK_s^2)\leq 0 ,\quad \text{for all}~~ t\in[0,T].
\ee
Then, plugging (\ref{minest}) in (\ref{itouni}) and taking expectation we obtain
\end{proof}
\b* 
\E[|Y_t^1-Y_t^2|^2]&+&\E[\Int_t^T \|Z_s^1-Z_s^2\|^2ds] \leq  2\E[\Int_t^T (Y_s^1-Y_s^2)^*(f_s(Y_s^1,Z_s^1)-f_s(Y_s^2,Z_s^2))ds]\nonumber\\
&+& \E[\Int_t^T \|h_s(Y_s^1,Z_s^1)-h_s(Y_s^2,Z_s^2)\|^2ds]. 
\e*
Hence from the Lipschitz Assumption on $h$ and the inequality $2ab\leq \epsilon a^2+\epsilon^{-1}b^2$, for all $\epsilon>0$, it follows that 
\b* 
\E[|Y_t^1-Y_t^2|^2]&+&\E[\Int_t^T \|Z_s^1-Z_s^2\|^2ds]\leq  (c+\epsilon) \E[\Int_t^T |Y_s^1-Y_s^2|^2ds]\\
&+&\epsilon^{-1}\E[\Int_t^T |f_s(Y_s^1,Z_s^1)-f_s(Y_s^2,Z_s^2)|^2ds]
+\alpha \E[\Int_t^T \|Z_s^1-Z_s^2\|^2ds],
\e*
where $0<\alpha <1$. Choosing $\epsilon=\Frac{2 c}{1-\alpha}$ and using the Lipschitz Assumption on $f$, we get
\b* 
\E[|Y_t^1-Y_t^2|^2]&+&\E[\Int_t^T \|Z_s^1-Z_s^2\|^2ds]\leq  (c+\Frac{2 c}{1-\alpha}+\Frac{1-\alpha}{2}) \E[\Int_t^T |Y_s^1-Y_s^2|^2ds]\\
&+&\Frac{1-\alpha}{2}\E[\Int_t^T \|Z_s^1-Z_s^2\|^2ds]
+\alpha \E[\Int_t^T \|Z_s^1-Z_s^2\|^2ds].
\e*
Consequently
$$\E[|Y_t^1-Y_t^2|^2]+(\Frac{1-\alpha}{2})\E[\Int_t^T \|Z_s^1-Z_s^2\|^2ds]\leq  (c+\Frac{2 c}{1-\alpha}+\Frac{1-\alpha}{2}) \E[\Int_t^T |Y_s^1-Y_s^2|^2ds].$$
From Gronwall's lemma, $\E[|Y_t^1-Y_t^2|^2] = 0,~ 0\leq t\leq T$, and $\E[\Int_0^T \|Z_s^1-Z_s^2\|^2ds] =0$.\\[0.3cm]
{\bf{b) Existence}}: The existence of a solution will be proved by penalisation method. For $n\in \mathbb{N}$, we consider for
all $t\in [0,T]$,
\begin{align}
Y_{t}^{n}=\xi+\Int_{t}^{T}f_s(Y_s^n,Z_s^n)ds+\Int_{t}^{T}h_s(Y_s^n,Z_s^n)d\W_s-n
\Int_{t}^{T}(Y_{s}^{n}- \pi(Y_s^n))ds
-\int_{t}^{T}Z_{s}^{n}dB_{s}.
\label{BDSDEpen}
\end{align}
From Pardoux and Peng \cite{pp1994} (Theorem 1.1), we know that the above equation admits a unique solution with coefficient $f^n$ and $h$, where $
f^{n}(t,x,y)=f(t,x,y,z)-n(y-\pi(y))$.\\
Denote by $K^n_t:=-n\Int_0^t(Y_{s}^{n}- \pi(Y_s^n))ds$. In order to prove the convergence of the sequence $(Y^n,Z^n,K^n)$ to the solution of our RBDSDE \eqref{existence:RBDSDE}, we need several lemmas.\\

\noindent We start with the following lemma:  
\begin{Lemma} \label{lem}There exists a constant $C>0$ such that
\be 
\forall n\in\N \qquad  \E[\Int_0^T d^2(Y_s^n,D)ds]\leq \Frac{C}{n}.
\ee
\end{Lemma}
\begin{proof}
For the sake of simplicity, we treat only the case where $D$ is a convex set with class $C^2$ boundary implying that the function $\rho(x)=d^2(x,D)=|x-\pi(x)|^2$ belongs to $C^2$ (see Subsection \ref{convexity:subsubsection} and Appendix \ref{Some properties of convexity}). Thus, we apply generalized It\^o's formula (Lemma 1.3 in \cite{pp1994} p.213) to $\rho(Y_t^n)$ to obtain
\be\begin{split} 
\rho(Y_t^n)&+ \Frac{1}{2}\Int_t^T trace[Z_s^nZ_s^{n *}Hess \rho(Y_s^n)]ds = \rho(\xi)
+ \Int_t^T (\nabla \rho(Y_s^n))^* f_s(Y_s^n,Z_s^n)ds\\&- \Int_t^T (\nabla \rho(Y_s^n))^* Z_s^n dB_s + \Int_t^T (\nabla \rho(Y_s^n))^* h_s(Y_s^n,Z_s^n)d\W_s\\
& + \Frac{1}{2}\Int_t^T trace[(h_sh_s^*)(Y_s^n,Z_s^n)Hess \rho(Y_s^n)]ds - 2n \Int_t^T (Y_s^n- \pi(Y_s^n))^*(Y_s^n-\pi(Y_s^n))ds.
\label{itorho}
\end{split}
\ee 
Since $\xi\in \bar{D} ~a.s.$, we have that $\rho(\xi)=0$. 
We get from the fact that $|\nabla\rho(x)|^2=4\rho(x)$ and the boundedness of $h$ and the Hessian of $\rho$
\be
\begin{split}
\rho(Y_t^n)&+ \Frac{1}{2}\Int_t^T trace[Z_s^nZ_s^{n *}Hess \rho(Y_s^n)]ds  + 2 n\Int_t^T d^2(Y_s^n,D)ds\\
&\leq 2 \Int_t^T (\rho(Y_s^n))^{1/2} |f_s(Y_s^n,Z_s^n)|ds - 2 \Int_t^T (Y_s^n- \pi(Y_s^n))^* Z_s^n dB_s\\
&+2 \Int_t^T (Y_s^n- \pi(Y_s^n))^* h_s(Y_s^n,Z_s^n)d\W_s  +C.
\end{split}
\ee
Now the inequality $2ab\leq a^2+b^2$ with $a=\sqrt{\Frac{n}{2}\rho(Y_s^n)}$ yields 
\b*
(\rho(Y_s^n))^{1/2} |f_s(Y_s^n,Z_s^n)|&\leq &\Frac{n}{4}\rho(Y_s^n)+\Frac{1}{n} |f_s(Y_s^n,Z_s^n)|^2.
\e*
Then it follows that, 
\be
\begin{split}
\rho(Y_t^n)&+ \Frac{1}{2}\Int_t^T trace[Z_s^nZ_s^{n *}Hess \rho(Y_s^n)]ds  + \Frac{3n}{2}\Int_t^T d^2(Y_s^n,D)ds\\
&\leq  2\Int_t^T\Frac{1}{n}|f_s(Y_s^n,Z_s^n)|^2ds - 2 \Int_t^T (Y_s^n- \pi(Y_s^n))^* Z_s^n dB_s \\
&+2 \Int_t^T (Y_s^n- \pi(Y_s^n))^* h_s(Y_s^n,Z_s^n)d\W_s
+ C.
\label{estY}
\end{split}
\ee
By taking expectation and using the boundedness of $f$, we have 
\be
\E[\rho(Y_t^n)]+\Frac{1}{2}\E[\Int_t^T trace[Z_s^nZ_s^{n *}Hess \rho(Y_s^n)]ds] + \Frac{3n}{2}\E[\Int_t^T d^2(Y_s^n,D)ds]\leq  C\big(1+\Frac{1}{n}\big).
\label{estdist}
\ee
Hence, the required result is obtained.
\ep 
\end{proof} 
\vspace{0.5cm}
\noindent The next lemma plays a crucial role to prove the strong convergence of $(Y^n,Z^n,K^n)$.  
\begin{Lemma}
\label{fundamental:lemma}
\be 
\E\Big[\underset{0\leq t \leq T}{\Sup}(d(Y_t^n,D))^4\Big]\underset{n\rightarrow +\infty}{\longrightarrow} 0.
\label{dist}
\ee
\end{Lemma}
\begin{proof} 
We denote by $\rho(x)=d^2(x,D)$ and $\varphi(x)= \rho^2(x)$.
By applying It\^o's formula to $\varphi(Y_t^n)=d^4(Y_t^n,D)$, we obtain that
\be\begin{split} 
\varphi(Y_t^n)&+ \Frac{1}{2}\Int_t^T trace[Z_s^nZ_s^{n *}Hess \varphi(Y_s^n)]ds = \varphi(\xi)
+ \Int_t^T (\nabla \varphi(Y_s^n))^* f_s(Y_s^n,Z_s^n)ds\\&- \Int_t^T (\nabla \varphi(Y_s^n))^* Z_s^n dB_s + \Int_t^T (\nabla \varphi(Y_s^n))^* h_s(Y_s^n,Z_s^n)d\W_s\\
& + \Frac{1}{2}\Int_t^T trace[(h_sh_s^*)(Y_s^n,Z_s^n)Hess \varphi_(Y_s^n)]ds - n \Int_t^T (\nabla \varphi(Y_s^n))^*(Y_s^n-\pi(Y_s^n))ds.
\label{ito}
\end{split}
\ee 
Since $\xi\in \bar{D} ~a.s.$, we have that $\varphi(\xi)=0$ and the chain rule of differentiation gives that
 \be
\nabla \varphi(x)&=&2\rho(x)\nabla\rho(x)=4\rho(x)(x-\pi(x)) \label{gradphi} \\
Hess \varphi(x)&=&2\nabla\rho(x)(\nabla\rho(x))^*+2\rho(x) Hess \rho(x)\label{hessphi}.
\ee
Then it follows that
\be\label{Ito}
\begin{split} 
\varphi(Y_t^n)&+ \Frac{1}{2}\Int_t^T trace[Z_s^nZ_s^{n *}Hess \varphi(Y_s^n)]ds = 4\Int_t^T (\rho(Y_s^n)(Y_s^n-\pi(Y_s^n))^* f_s(Y_s^n,Z_s^n)ds\\&- 4\Int_t^T (\rho(Y_s^n)(Y_s^n-\pi(Y_s^n))^* Z_s^n dB_s + 4\Int_t^T (\rho(Y_s^n)(Y_s^n-\pi(Y_s^n))^* h_s(Y_s^n,Z_s^n)d\W_s\\
& + \Frac{1}{2}\Int_t^T trace[(h_sh_s^*)(Y_s^n,Z_s^n)Hess \varphi(Y_s^n)]ds - 4n \Int_t^T \rho^2(Y_s^n)ds.
\end{split}
\ee 
By taking expectation we have 
\be
\begin{split}\label{estvarphi} 
\E[\varphi(Y_t^n)]&+ \Frac{1}{2}\E\big[\Int_t^T trace[Z_s^nZ_s^{n *}Hess \varphi(Y_s^n)]ds\big]+4n \E\big[\Int_t^T \varphi(Y_s^n)ds\big]\\
& = 4\E\big[\Int_t^T (\rho(Y_s^n)(Y_s^n-\pi(Y_s^n))^* f_s(Y_s^n,Z_s^n)ds]\\
 &+ \Frac{1}{2}\E\Big[\Int_t^T trace[(h_sh_s^*)(Y_s^n,Z_s^n)Hess \varphi(Y_s^n)]ds].
\end{split}
\ee
For the last term, we get from the fact that $|\nabla\rho(x)|^2=4\rho(x)$ and the boundedness of $h$ and $Hess\rho$ 
\be \label{esth}
\begin{split}
\E\Big[\Int_t^T trace[(h_sh_s^*)(Y_s^n,Z_s^n)&Hess \varphi(Y_s^n)]ds]\leq 2\E\Big[\Int_t^T\langle h_s(Y_s^n,Z_s^n),\nabla\rho(Y_s^n)\rangle^2 ds\Big]\\
& + \E\Big[\Int_t^T 2\rho(Y_s^n)trace[(h_sh_s^*)(Y_s^n,Z_s^n)Hess \rho(Y_s^n)]ds\Big]\\
&\leq C\E\Big[\Int_t^T|\nabla\rho(Y_s^n)|^2 ds\Big] + C\E\Big[\Int_t^T \rho(Y_s^n)ds\Big]\\
&\leq C  \E\Big[\Int_0^T(d(Y_s^n,D))^{2}ds\Big].
\end{split}
\ee 
Now the inequality $2ab\leq a^2+b^2$ with $a=(d(Y_s^n,D))^{2}$ and the boundedness of $f$ yield 
\be
\begin{split}\label{estf}
4(d(Y_s^n,D))^{3} |f_s(Y_s^n,Z_s^n)|&\leq 2 (d(Y_s^n,D))^{4}+ 2 (d(Y_s^n,D))^{2}|f_s(Y_s^n,Z_s^n)|^2\\
&\leq 2 \varphi(Y_s^n) + 2C(d(Y_s^n,D))^{2} .
\end{split}
\ee
By plugging the estimate (\ref{estf}) and \eqref{esth} in (\ref{estvarphi}), we obtain thanks to lemma \ref{lem}
\be\label{estphi}
\begin{split} 
\E[\varphi(Y_t^n)]+ \Frac{1}{2}\E\big[\Int_t^T trace [Z_s^nZ_s^{n *}& Hess \varphi(Y_s^n)]ds\big]+(4n-2) \E\big[\Int_t^T \varphi(Y_s^n)ds\big]\\
& \leq C\E\big[\Int_0^T (d(Y_s^n,D))^{2}ds\big]\leq C\big(\Frac{1}{n}+\Frac{1}{n^2}\big).
\end{split}
\ee
Notice also that  Hessian of $ \varphi(Y_s^n)$ is a positive semidefinite matrix since $ \varphi$ is a convex function, so we get that
$ \E\big[\Int_t^T trace [Z_s^nZ_s^{n *} Hess \varphi(Y_s^n)]ds\big] \geq 0$ and consequently,
\be\label{unifY}
\underset{0\leq t\leq T}{\Sup}\E[\varphi(Y_t^n)]\leq C\big(\Frac{1}{n}+\Frac{1}{n^2}\big).
\ee
Moreover, we can deduce from \eqref{estphi} that, for every $t\in[0,T]$
\be\label{estZ}
 \E\big[\Int_0^T trace [Z_s^nZ_s^{n *} Hess \varphi(Y_s^n)]ds\big]\longrightarrow 0, \, \text{as}\, n\rightarrow \infty.
\ee
On the other hand, taking the supremum over $t$ in the equation \eqref{Ito}, by Burkholder-Davis-Gundy's inequlity and the previous calculations it follows that
\be\label{uniformestimate}
\begin{split}
\E[\underset{0\leq t\leq T}{\Sup}\varphi(Y_t^n)]&\leq C \E[\Int_0^T\varphi(Y_s^n) ds]+C\E\Big[\Int_0^T (d(Y_s^n,D))^{2}ds\Big]\\
&+C\E\Big[\underset{0\leq t\leq T}{\Sup}\Int_t^T (\rho(Y_s^n)\nabla\rho(Y_s^n))^* Z_s^n dB_s\Big]\\
&+ C \E\Big[\underset{0\leq t\leq T}{\Sup}\Int_t^T (\rho(Y_s^n)\nabla\rho(Y_s^n))^* h_s(Y_s^n,Z_s^n)d\W_s\Big]\\
&\leq C \E[\Int_0^T\varphi(Y_s^n) ds]+C\E\Big[\Int_0^T (d(Y_s^n,D))^{2}ds\Big]\\&+C\E\Big[\Big(\Int_0^T (\rho(Y_s^n))^2\langle\nabla\rho(Y_s^n), Z_s^n\rangle^2 ds\Big)^{1/2}\Big]\\
&+ C\E\Big[\Big(\Int_0^T (\rho(Y_s^n))^2\langle\nabla\rho(Y_s^n), h_s(Y_s^n,Z_s^n)\rangle^2 ds\Big)^{1/2}\Big].
\end{split}
\ee
From the boundedness of $h$ and the fact that $|\nabla\rho(x)|^2=4\rho(x)$, we have
\be \label{estimate1}
\begin{split}
\E\Big[\Big(\Int_0^T (\rho(Y_s^n))^2&\langle\nabla\rho(Y_s^n)), h_s(Y_s^n,Z_s^n)\rangle^2 ds\Big)^{1/2}\Big]\leq C \E\Big[\Big(\Int_0^T (\rho(Y_s^n))^2\rho(Y_s^n) ds\Big)^{1/2}\Big]\\
&\leq C \E\Big[\underset{0\leq s\leq T}{\Sup}\big(\varphi(Y_s^n)\big) ^{1/2}\Big(\Int_0^T \rho(Y_s^n) ds\Big)^{1/2}\Big]\\
&\leq  \Frac{1}{4}\E\Big[\underset{0\leq s\leq T}{\Sup}\varphi(Y_s^n)\Big]+C^2 \E\Big[\Int_0^T (d(Y_s^n,D))^2ds\Big].
\end{split}
\ee
By the Holder's inequality, we obtain
\be \label{estimate2}
\begin{split}
\E\Big[\Big(\Int_0^T (\rho(Y_s^n))^2\langle\nabla\rho(Y_s^n)), Z_s^n\rangle^2 ds\Big)^{1/2}\Big]&\leq C \E\Big[\underset{0\leq s\leq T}{\Sup}\big(\varphi(Y_s^n)\big) ^{1/2}\Big(\Int_0^T\langle\nabla\rho(Y_s^n)), Z_s^n\rangle^2 ds\Big)^{1/2}\Big]\\
&\leq  \Frac{1}{4}\E\Big[\underset{0\leq s\leq T}{\Sup}\varphi(Y_s^n)\Big]+C^2 \E\Big[\Int_0^T\langle\nabla\rho(Y_s^n), Z_s^n\rangle^2 ds\Big].
\end{split}
\ee
Substituting \eqref{estimate1} and \eqref{estimate2} in \eqref{uniformestimate} leads to
\be \label{estimate3}
\begin{split}
\E[\underset{0\leq t\leq T}{\Sup}\varphi(Y_t^n)]&\leq C \E[\Int_0^T\varphi(Y_s^n) ds]+C\E\Big[\Int_0^T (d(Y_s^n,D))^{2}ds\Big]\\
&+C^2 \E\Big[\Int_0^T\langle\nabla\rho(Y_s^n), Z_s^n\rangle^2 ds\Big].
\end{split}
\ee
In the other hand, Hessian of $ \rho(Y_s^n)$ is a positive semidefinite matrix since $ \rho$ is a convex function, so we get that
$ \E\big[\Int_t^T trace [Z_s^nZ_s^{n *} \rho(Y_s^n)Hess \rho(Y_s^n)]ds\big] \geq 0$. By the equation \eqref{hessphi}, we can deduce that 
\be
2\E\Big[\Int_0^T\langle\nabla\rho(Y_s^n), Z_s^n\rangle^2 ds\Big]\leq \E\Big[\Int_0^T trace[Z_s^nZ_s^{n *} Hess \varphi(Y_s^n)]ds],
\ee   and from \eqref{estZ}, we get 
\b*
\E\Big[\Int_0^T\langle\nabla\rho(Y_s^n), Z_s^n\rangle^2 ds\Big]\longrightarrow 0 \, \text{as}\, n\rightarrow\infty.
\e*

Finally, by using \eqref{unifY}, \eqref{estimate3} and Lemma \ref{lem}, we get the desired result.\ep
\end{proof}
\begin{Remark}
Contrary to Gegout-Petit and Pardoux \cite{GPP95}, we have to prove the fundamental lemma with power $4$ instead of $2$. In fact, if we apply the genelized It\^o's formula to $\rho(Y_t^n)=d^2(Y_t^n,D)$, we obtain that
\be\begin{split} 
\rho(Y_t^n)&+ \Frac{1}{2}\Int_t^T trace[Z_s^nZ_s^{n *}Hess \rho(Y_s^n)]ds = \rho(\xi)
+ \Int_t^T (\nabla \rho(Y_s^n))^* f_s(Y_s^n,Z_s^n)ds\\&- \Int_t^T (\nabla \rho(Y_s^n))^* Z_s^n dB_s + \Int_t^T (\nabla \rho(Y_s^n))^* h_s(Y_s^n,Z_s^n)d\W_s\\
& + \Frac{1}{2}\Int_t^T trace[(h_sh_s^*)(Y_s^n,Z_s^n)Hess \rho(Y_s^n)]ds - n \Int_t^T (\nabla \rho(Y_s^n))^*(Y_s^n-\pi(Y_s^n))ds.
\end{split}
\ee
To prove the fundamental lemma, we need to estimate all the terms in the right hand side of the  above equation in terms of quantities which depend on $n$. But, since $Hess\rho$ is bounded we get 
\b* 
\begin{split}
\E\Big[\Int_t^T trace[(h_sh_s^*)(Y_s^n,Z_s^n)Hess \rho(Y_s^n)]ds]\leq C \E\Big[\Int_t^T\| h_s(Y_s^n,Z_s^n)\|^2 ds\Big].
\end{split}
\e*
Then, there is no hope to obtain the required convergence to $0$.
\end{Remark}
\begin{Lemma} \label{conv}
The sequence $(Y^n,Z^n)$ is a Cauchy sequence in ${\mathcal S}^2_{k}([0,T]) \times {\mathcal H}^2_{k\times d}([0,T])$, i.e.
$$\E[\underset{0\leq t\leq T}{\Sup}|Y_t^n- Y_t^m|^2 + \Int_0^T \|Z_t^n-Z_t^m\|^2dt]\longrightarrow 0 \,\text{as}\,\, n,m \rightarrow +\infty.$$
\end{Lemma}
\begin{proof}
For all $n,m\geq 0$, we apply It\^o formula to $|Y_t^n-Y_t^m|^2$
\be \label{itoexis}
\begin{split}
|Y_t^n-Y_t^m|^2&+\Int_t^T \|Z_s^n-Z_s^m\|^2ds  =  2\Int_t^T (Y_s^n-Y_s^m)^*(f_s(Y_s^n,Z_s^n)-f_s(Y_s^m,Z_s^m))ds\\
& +2\Int_t^T (Y_s^n-Y_s^m)^*(h_s(Y_s^n,Z_s^n)-h_s(Y_s^m,Z_s^m))d\W_s - 2\Int_t^T (Y_s^n-Y_s^m)(Z_s^n-Z_s^m)dB_s\\
& +\Int_t^T \|h_s(Y_s^n,Z_s^n)-h_s(Y_s^m,Z_s^m)\|^2ds- 2n\Int_t^T (Y_s^n-Y_s^m)^*(Y_s^n-\pi(Y_s^n))ds\\
&+2m\Int_t^T (Y_s^n-Y_s^m)^*(Y_s^m-\pi(Y_s^m))ds.
\end{split}
\ee
By the property (\ref{prop2}), we have 
\begin{eqnarray}
\begin{split}
- 2n\Int_t^T (Y_s^n-Y_s^m)^*(Y_s^n-\pi(Y_s^n))ds &+ 2m\Int_t^T (Y_s^n-Y_s^m)^*(Y_s^m-\pi(Y_s^m))ds\\
&\leq 2(n+m)\Int_t^T (Y_s^n-\pi(Y_s^n))^*(Y_s^m-\pi(Y_s^m))ds.\\
&
\end{split}
\end{eqnarray}
Hence, from the Lipschitz continuity on $f$ and $h$, and taking expectation yields
\begin{eqnarray}\label{itoestimate}
\begin{split}
\E[|Y_t^n-Y_t^m|^2] + \E[\Int_t^T \|Z_s^n-Z_s^m\|^2ds]&\leq  2\E[\Int_t^T C(|Y_s^n-Y_s^m|^2+ |Y_s^n-Y_s^m|\|Z_s^n-Z_s^m\|)ds]\\
&+\E[\Int_t^T C(|Y_s^n-Y_s^m|^2+ \alpha\|Z_s^n-Z_s^m\|{\color{blue}{^2}})ds]\\
&+ 2(n+m)\E[\Int_t^T (Y_s^n-\pi(Y_s^n))^*(Y_s^m-\pi(Y_s^m))ds].
\end{split}
\end{eqnarray}
For the last term, we need the following lemma whose proof is postponed in the Appendix.
\begin{Lemma}\label{extraestimate}
There exists a constant $C>0$ such that, for each $n\geq 0$,
\be 
 \E\big[\Big(n\Int_0^T d(Y_s^n,D) ds\Big)^2\big]\leq C
 \label{boundvariation}
\ee
\end{Lemma}
\vspace{0.5cm}
\noindent Now we can deduce from the H\"older inequality and Lemma \ref{extraestimate} that 
\be \label{estdistance}
\begin{split} 
n\E[\Int_t^T &(Y_s^n-\pi(Y_s^n))^*(Y_s^m-\pi(Y_s^m))ds] \leq  n\E[\Int_t^T  d(Y_s^n,D)d(Y_s^m,D))ds]\\
&\leq n\E[\underset{0\leq s\leq T}{\Sup}d(Y_s^m,D)\Int_t^T  d(Y_s^n,D)ds)]\\
 &\leq \Big(\E\big[\Big(n\Int_0^T d(Y_s^n,D) ds\Big)^2\big]\Big)^{1/2} \Big(\E[\underset{0\leq s\leq T}{\Sup}d^2(Y_s^m,D)]\Big)^{1/2}\\
&\leq C\Big(\E[\underset{0\leq s\leq T}{\Sup}d^2(Y_s^m,D)]\Big)^{1/2}.
\end{split}
\ee
Substituting (\ref{estdistance}) in the previous inequality \eqref{itoestimate}, we have
\b*
\begin{split}
\E[|Y_t^n-Y_t^m|^2] &+ (1-\alpha-C\gamma)\E[\Int_t^T \|Z_s^n-Z_s^m\|^2ds]\leq  C(1+\Frac{1}{\gamma})\E[\Int_t^T |Y_s^n-Y_s^m|^2ds]\\
&+ C\Big(\E[\underset{0\leq s\leq T}{\Sup}d^2(Y_s^n,D)]\Big)^{1/2}+ C\Big(\E[\underset{0\leq s\leq T}{\Sup}d^2(Y_s^m,D)]\Big)^{1/2}.
\end{split}
\e*
Choosing $1-\alpha-C\gamma>0$, by Gronwall's lemma, we obtain
\be
\underset{0\leq t\leq T}{\Sup}\E[|Y_t^n-Y_t^m|^2]\leq C\Big(\E[\underset{0\leq s\leq T}{\Sup}d^2(Y_s^m,D)]\Big)^{1/2}+ C\Big(\E[\underset{0\leq s\leq T}{\Sup}d^2(Y_s^n,D)]\Big)^{1/2}.
\ee
We deduce similarly
\be
 \E[\Int_0^T \|Z_s^n-Z_s^m\|^2ds]\leq C\Big(\E[\underset{0\leq s\leq T}{\Sup}d^2(Y_s^m,D)]\Big)^{1/2}+ C\Big(\E[\underset{0\leq s\leq T}{\Sup}d^2(Y_s^n,D)]\Big)^{1/2}.
\label{estunifZ}
\ee
Next, by \eqref{itoexis}, the Burkholder-Davis-Gundy and the Cauchy-Schwarz inequalities we get 
\begin{align*}
&\E[\underset{0\leq t\leq T}{\Sup}|Y_t^n-Y_t^m|^2]\leq C\E[\Int_0^T | Y_s^n-Y_s^m||f(s,Y_s^n,Z_s^n)-f(s,Y_s^m,Z_s^m)|ds]\\
&+C\E\big(\Int_0^T | Y_s^n-Y_s^m|^2\|h_s(Y_s^n,Z_s^n)-h_s(Y_s^m,Z_s^m)\|^2ds\big)^{1/2}+C\E\big(\Int_0^T | Y_s^n-Y_s^m|^2\| Z_s^n-Z_s^m\|^2ds\big)^{1/2}\\&+\E[\Int_0^T C(|Y_s^n-Y_s^m|^2+ \alpha\|Z_s^n-Z_s^m\|^2)ds]+  2(n+m)\E[\Int_0^T (Y_s^n-\pi(Y_s^n))^*(Y_s^m-\pi(Y_s^m))ds].
\end{align*}
Then, it follows by the Lipschitz Assumption \ref{Ass2} on $f$ and $h$ and  (\ref{estdistance}) that for any $n,m\geq 0$
\b*
\begin{split}
\E[\underset{0\leq t\leq T}{\Sup}|Y_t^n-Y_t^m|^2]&\leq  C\Big(\E[\underset{0\leq s\leq T}{\Sup}d^2(Y_s^m,D)]\Big)^{1/2}+C\Big(\E[\underset{0\leq s\leq T}{\Sup}d^2(Y_s^n,D)]\Big)^{1/2}\\
&+ C\E(\underset{0\leq t\leq T}{\Sup}|Y_t^n-Y_t^m|^2\Int_0^T \|Z_s^n-Z_s^m\|^2ds)^{1/2}\\
&\leq  C\Big(\E[\underset{0\leq s\leq T}{\Sup}d^2(Y_s^m,D)]\Big)^{1/2}+C\Big(\E[\underset{0\leq s\leq T}{\Sup}d^2(Y_s^n,D)]\Big)^{1/2}\\
&+ C\varepsilon\E(\underset{0\leq t\leq T}{\Sup}|Y_t^n-Y_t^m|^2)+ C\varepsilon^{-1}\E(\Int_0^T \|Z_s^n-Z_s^m\|^2ds).
\end{split}
\e*
Choosing $1-C\varepsilon>0$ and from the inequality (\ref{estunifZ}) we conclude that
\b*
\begin{split}
\E[\underset{0\leq t\leq T}{\Sup}|Y_t^n-Y_t^m|^2] &\leq C\Big(\E[\underset{0\leq s\leq T}{\Sup}d^2(Y_s^m,D)]\Big)^{1/2}+C\Big(\E[\underset{0\leq s\leq T}{\Sup}d^2(Y_s^n,D)]\Big)^{1/2}\\
&\leq C\Big(\E[\underset{0\leq s\leq T}{\Sup}d^4(Y_s^m,D)]\Big)^{1/4}+C\Big(\E[\underset{0\leq s\leq T}{\Sup}d^4(Y_s^n,D)]\Big)^{1/4}\longrightarrow 0,
\end{split}
\e*
as $n,m\rightarrow \infty$, where Lemma \ref{fundamental:lemma} has been used.\ep
\end{proof}
\vspace{0.6cm}
\noindent Finally, we conclude that $(Y^n,Z^n)$ is a Cauchy sequence in ${\Sc}^2_k([0,T])\times {\Hc}^2_{k\times d}([0,T])$ and therefore there exists a unique pair $(Y_t,Z_t)$  of $\Fc_t$- measurable processes which valued in $\R^k\times\R^{k\times d}$, satisfying
\be
\E(\underset{0\leq t\leq T}{\Sup} |Y_t^n-Y_t|^2 + \Int_0^T |Z_t^n-Z_t|^2 dt) \rightarrow 0 \quad \mbox{as}~~ n\rightarrow\infty.
\label{converg}
\ee
Consequently, since for any $n\geq 0$ and $0\leq t\leq T$,
\be
\begin{split}
K_t^n-K_t^m &= Y_0^n-Y_0^m-Y_t^n-Y_t^m-\Int_0^t (f_s(Y_s^n,Z_s^n)-f_s(Y_s^m,Z_s^m))ds\\
&-\Int_0^t (h_s(Y_s^n,Z_s^n)-h_s(Y_s^m,Z_s^m))d\W_s
 +\Int_0^t (Z_s^n-Z_s^m) dB_s.
\end{split}
\ee
we obtain from (\ref{converg}) and Burkholder-Davis-Gundy inequality,
\be
\E(\underset{0\leq t\leq T}{\Sup} |K_t^n-K_t^m|^2)\rightarrow 0 \quad \mbox{as}~~ n, m\rightarrow\infty.
\ee
Hence, there exists a $\Fc_t$- adapted continuous process $(K_t)_{0\leq t\leq T}$ ( with $K_0=0$) such that $$\E(\underset{0\leq t\leq T}{\Sup} |K_t-K_t^n|^2 )\rightarrow 0 \quad \mbox{as}~~ n\rightarrow\infty.$$
Furthermore, \eqref{boundvariation} shows that the total variation of $K^n$ is uniformly bounded. Thus, $K$ is also of uniformly bounded variation.
Passing to the limit in \eqref{BDSDEpen}, the processes $(Y_t,Z_t,K_t)_{0\leq t\leq T}$ satisfy 
$$Y_t  = \xi +\Int_t^T f_s(Y_s,Z_s)ds +\Int_t^T h_s(Y_s,Z_s)d\W_s -\Int_t^T Z_sdB_s +K_T-K_t ~,~ 0\leq t\leq T.$$
Since we have from Lemma \ref{fundamental:lemma} that $Y_t$ is in $\bar{D}$, it remains to check the minimality property for $ (K_t)$, namely i.e., for any continuous progressively measurable process $(z_t)$ valued in $\bar{D}$, $$\Int_0^T (Y_t-z_t)^* dK_t \leq 0.$$
We note that (\ref{prop1}) gives us 
$$\Int_0^T (Y_t^n-z_t)^*dK_t^n= -n\Int_0^T (Y_t^n-z_t)^*(Y_t^n-\pi(Y_t^n))dt \leq 0.$$
Therefore, we will show that we can extract a subsequence such that $\Int_0^T (Y_t^n-z_t)^* dK_t^n$ converge a.s. to $\Int_0^T (Y_t-z_t)^* dK_t.$ Following the proof of Lemma \ref{estapriori} in Appendix, we have
\be
2\gamma \|K^n\|_{VT}&\leq & |\xi -a|^2 +2\Int_0^T (Y_s^n-a)f_s(Y_s^n,Z_s^n)ds + 2\Int_0^T  (Y_s^n-a) h_s(Y_s^n,Z_s^n)d\W_s\nonumber\\
&+& \Int_0^T  \|h_s(Y_s^n,Z_s^n)\|^2ds - 2\Int_0^T(Y_s^n-a)Z_s^n dB_s.
\ee
Notice that the right hand side tends in probability as $n$ goes to infinity  to $$|\xi -a|^2 +2\Int_0^T (Y_s-a)f_s(Y_s,Z_s)ds + 2\Int_0^T  (Y_s-a) h_s(Y_s,Z_s)d\W_s + \Int_0^T \|h_s(Y_s,Z_s)\|^2ds - 2\Int_0^T(Y_s-a)Z_s dB_s.$$ 
Thus, there exists a subsequence $(\phi(n))_{n\geq 0}$ such that the convergence is almost surely and $\|K^{\phi(n)}\|_{VT}$ is bounded. Moreover, due to the convergence in $\L^2$ of $\underset{0\leq t\leq T}{\Sup} |Y_t^n-Y_t|^2$ to $0$, we can extract a subsequence  from $(\phi(n))_{n\geq 0}$ such that $Y^{\phi(\psi(n))}$ converges uniformly to $Y$. Hence, we apply  Lemma 5.8 in \cite{GPP95} and we obtain
$$\Int_0^T(Y_t^{\phi(\psi(n))}-z_t)^*dK_t^{\phi(\psi(n)))}\longrightarrow\Int_0^T(Y_t -z_t)^*dK_t\quad a.s. ~~\mbox{as}~~ n\rightarrow\infty$$
which is  the required result.
\ep
\section{Weak solution of semilinear SPDE in a convex domain}
\label{section:SPDE}
The aim of this section is to give a Feynman-Kac's formula for the weak solution of a  semilinear reflected SPDEs  \eqref{OSPDE1}  in a given convex domain $D$
via Markovian class of RBDSDEs  studied in the last section.  As explained in the introduction, the solution  of such SPDEs  is expressed as a pair  $(u,\nu) $ where 
$u$ is a predictable continuous process which takes values in a  Sobolev space and $\nu$ is a signed Radon regular measure. The bounded variation processes $K $ component of the  solution of the reflected BDSDE  controls the set when   $ u$ reaches the boundary of $D$. In fact, this bounded variation process determines the measure $\nu$ from a particular relation by using the inverse of the flow associated to the  diffusion operator. 
\subsection{Notations and Hypothesis}
Let us first introduce some notations:\\
- $C^n_{l,b}(\R^p,\R^q)$ the set of $C^n$-functions which grow at most linearly at infinity and whose partial derivatives of order less than or equal to $n$ are bounded.\\
- $\mathbf{L}_{w}^2\left( \mathbb{{R}}^d\right) $ will be a Hilbert with the inner product,
$$ \left( u,v\right)_{w} =\int_{\mathbb{R}^d}u\left( x\right) v\left(
x\right) w (x) dx,\;\left\| u\right\| _2=\left(
\int_{\mathbb{R}^d}u^2\left( x\right) w (x) dx\right) ^{\frac
12}. $$ 
\vspace{0.1cm}
\noindent
\begin{Assumption}\label{assweight}
We assume that $ w$ is the weight function that satisfy the following conditions:
\begin{itemize}
\item $w$ is a continuous positive function.
\item $w$ is integrable and $ \Frac{1}{w}$ is locally integrable.
\end{itemize}
\end{Assumption}
\vspace{0.2cm}
\noindent In general, we shall use for the usual $L^2$-scalar product
$$(u,v)=\Int_{\mathbb{R}^d} u(x)v(x)\, dx,$$
where $u$, $v$ are measurable functions defined in $\mathbb{R }^d$
and $uv \in \mathbf{L}^1 (\mathbb{R}^d )$.\\
Our evolution problem will be considered over a fixed time interval
$[0,T]$ and the norm for an element of $\mathbf{L}_{w}^2\left(
[0,T] \times \mathbb{{R}}^d\right) $ will be denoted by
$$\left\| u\right\| _{2,2}=\left(\Int_0^T  \int_{\mathbb{R}^d} |u (t,x)|^2 w(x)dx dt \right)^{\frac 12}. $$
We assume the following hypotheses :
\begin{Assumption}\label{assxi}
$\Phi:\R^d\rightarrow\R^k$ belongs to $\mathbf{L}_{w}^2(\R^d)$ and $\Phi(x)\in\bar{D}~~ a.e. ~\forall x\in\R^d$;
\end{Assumption}
\begin{Assumption} \label{assgener}
\begin{itemize} 
\item[\rm{(i)}] $f:[0,T]\times \R^d\times \R^k\times \R^{k\times
d}\rightarrow\R{\color{blue}{^k}}$ and $h:[0,T]\times \R^d\times \R^k\times \R^{k\times
d}\rightarrow\R^{k\times l}$ are measurable in $(t,x,y,z)$ and
satisfy $ f^0, h^0 \in \mathbf{L}_{w}^2\left( [0,T] \times
\mathbb{{R}}^d\right) $ where $f_t^0 (x) := f (t,x,0, 0)$, $h_t^0 := h (t,
,x,0, 0)$.
\item[\rm{(ii)}] There exist constants $c>0$ and $0<\alpha<1$ such that for any $(t,x)\in[0,T]\times\R^d~;~
(y_1,z_1),(y_2,z_2)\in\R^k\times\R^{k\times d}$
\b*
|f_t(x,y_1,z_1)-f_t(x,y_2,z_2)|^2 &\leq & c\big(|y_1-y_2|^2+\|z_1-z_2\|^2\big)\\
\|h_t(x,y_1,z_1)-h_t(x,y_2,z_2)\|^2 &\leq & c|y_1-y_2|^2+\alpha \|z_1-z_2\|^2.
\e*
\end{itemize}
\end{Assumption}
\begin{Assumption}\label{assdiff}
The coefficients $b$ and $\sigma$ of the second order differential operator $\Lc$ (\ref{operator conv}) satisfy $b\in C^2_{l,b}(\R^d;\R^d)$, $\sigma\in C^3_{l,b}(\R^d;\R^{d\times d}).$
\end{Assumption}
\begin{Assumption}\label{assint}
\begin{itemize}
\item[\rm{(i)}] $\Phi\in \mathbf{L}_{w}^4(\R^d).$
\item[\rm{(ii)}] There exist $c>0$ and  $0\leq \beta < 1$ such that for all $(t,x,y,z)\in [0,T]\times\R^k\times\R^{k\times d}$
$$h_t \, h_t^* (x,y,z)\leq c (Id_{\R^k}+ yy^*)+ \beta  \, zz^*.$$
\item[\rm{(iii)}] $f$ and $h$ are uniformly bounded in $(x,y,z)$.
\end{itemize}
\end{Assumption}
\subsection{Weak formulation for a solution of Stochastic  PDEs}
\label{definition:solution}
The space of test functions which we employ in the definition of
weak solutions of the evolution equations  \eqref{SPDE1} is $
\mathcal{D}_T  = \mathcal{C}^{\infty} (\left[0,T]\right) \otimes
\mathcal{C}_c^{\infty} \left(\mathbb{R}^d\right)$, where
\begin{itemize}
\item $\mathcal{C}^{\infty} \left([0,T]\right)$ denotes the space of real
functions which can be extended as infinitely differentiable functions
in the neighborhood of $[0,T]$, 
\item $\mathcal{C}_c^{\infty}\left(\mathbb{R}^d\right)$ is the space of
infinite differentiable functions with compact supports in
$\mathbb{R}^d$.
\end{itemize}
We denote by  $ {\mathcal H}_T$ the space of  $ \Fc_{t,T}^W$-progressively measurable  processes  $(u_t ) $ with values  in the weighted Sobolev space $ H_{w} ^1 (\R^d)$ where 
$$ H_{w} ^1 (\R^d):=\{v \in \mathbf{L}_{w}^2(\R^d) \; \big|\; \nabla v\sigma\in \mathbf{L}_{w}^2(\R^d)\} $$
endowed with the norm
$$\begin{array}{ll}
\|u\|_{{\mathcal H}_T}^2=
 \E \,  \big[\underset{ 0 \leq t \leq T}{\Sup} \|u_s \|_2^2 +   \Int_{ \mathbb{R}^d} \Int_0^T  |\nabla
u_s (x)\sigma(x)|^2 ds w(x)dx \big],
\end{array}
$$
where we denote the gradient by $\nabla u (t,x) = \big(\partial_1 u
(t,x), \cdot \cdot \cdot, \partial_d u (t,x) \big)$. Here, the derivative is defined in the weak sense (Sobolev sense).
\begin{Definition}[{\textbf{Weak solution of regular SPDE}}]
We say that $ u \in \mathcal{H}_T $ is a Sobolev solution of SPDE  $\left(
\ref{SPDE1}\right) $ if the following
relation holds, for each $\varphi \in \mathcal{D}_T ,$
\begin{equation}\label{wspde1}
\begin{array}{ll}
\displaystyle\int_t^T(u(s,x),\partial_s\varphi(s,x))ds+(u(t,x),\varphi(t,x))-(\Phi(x),\varphi(T,x))
-\int_t^T( u(s,x),\mathcal L^\ast \varphi(s,x))ds
\\=\displaystyle\int_t^T(f_s(x,u(s,x),\nabla u(s,x)\sigma(x)),\varphi(s,x))ds +\displaystyle\int_t^T(h_s(x,u(s,x),\nabla u(s,x)\sigma(x)),\varphi(s,x))d\W_s.
\end{array}
\end{equation}
where ${\mathcal L}^\ast$ is the adjoint operator of ${\mathcal L}$.
We denote by $ u:=\mathcal{U }(\Phi, f,h)$ the solution of SPDEs with data $(\Phi,f,h)$.
\end{Definition}
\vspace{0.2cm}
\noindent The existence and uniqueness  of weak solution  for SPDEs \eqref{wspde1} is ensured by Theorem 3.1 in Bally and Matoussi \cite{BM01} or Denis and Stoica \cite{DS04}.
\subsection{Stochastic flow of diffeomorphisms and random test functions}
\label{Flow} We are concerned in this part with solving our problem by developing a stochastic flow method which was first introduced in Kunita \cite{K84}, \cite{K90} and Bally, Matoussi
\cite{BM01}. We recall that  $\{X_{t,s}(x), t\leq s\leq T\}$ is the diffusion process starting from $x$ at time $t$ and is the strong solution  of the equation:
 \begin{equation}\label{sde}
  X_{t,s}(x)=x+\Int_{t}^{s}b(X_{t,r}(x))dr+\Int_{t}^{s}\sigma(X_{t,r}(x))dB_r.
\end{equation}
The existence and uniqueness of this solution was proved in Kunita \cite{K84}. Moreover, we have the following properties:
\begin{Proposition}\label{estimatesde}
Under Assumption \ref{assdiff} and for each $t>0$, there exists a version of $\{X_{t,s}(x);\,x\in
\R^d,\,s\geq t\}$ such that $X_{t,s}(\cdot)$ is a $C^2(\R^d)$-valued
continuous process which satisfi{\color{blue}{es}} the flow property: $X_{t,r}(x)=X_{s,r} \circ X_{t,s} (x)$, $0\leq t<s<r$.
Furthermore,
for all $p\geq 2$, there exists $M_p$ such that for all $0\leq t<s$,
$x,x'\in\R^d$, $h,h'\in\R\backslash{\{0\}}$,
$$\begin{array}{ll}
\E(\underset{t\leq r\leq s}{\Sup}|X_{t,r}(x)-x|^p)\leq
M_p(s-t)(1+|x|^p),\\
\E(\underset{t\leq r\leq s}{\Sup}|X_{t,r}(x)-X_{t,r}(x')-(x-x')|^p)\leq
M_p(s-t)(|x-x'|^p),\\
\E(\underset{t\leq r\leq s}{\Sup}|\Delta_h^i[X_{t,r}(x)-x]|^p)\leq
M_p(s-t),\\
\E(\underset{t\leq r\leq s}{\Sup}|\Delta_h^i X_{t,r}(x)-\Delta_{h'}^i
X_{t,r}(x')|^p)\leq M_p(s-t)(|x-x'|^p+|h-h'|^p),
\end{array}$$
where $\Delta_h^ig(x)=\frac{1}{h}(g(x+he_i)-g(x))$, and
$(e_1,\cdots,e_d)$ is an orthonormal basis of $\R^d$.
\end{Proposition}

\noindent {  For  all $v\in \mathbf{L}^2(\R^d)$, the process $(v\circ X_{t,s} (x))_{t \leq s \leq T}$ defined as the composition of $v$
with the stochastic flow  $ (X_{t,s}(x))_{t \leq s \leq T}$  is   well defined thanks to  the  following  equivalence norms results
(see  \cite{BL97} and \cite{BM01} for the proofs. 
\begin{Proposition}\label{equivalence:normes conv}
There exists two constants $c>0$ and $C>0$ such that for every
$t\leq s\leq T$ and $\varphi\in L^1(\R^d,dx)$,
\begin{equation}\label{equi1 conv} c\Int_{\R^d}|\varphi(x)|w(x)dx\leq
\Int_{\R^d}\E(|\varphi(X_{t,s}(x))|)w(x)dx\leq
C\Int_{\R^d}|\varphi(x)|w(x)dx. 
\end{equation} 
Moreover, for
every $\Psi\in L^1([0,T]\times\R^d,dt\otimes dx)$,
\begin{equation} \label{equi2 conv}
c\Int_{\R^d}\Int_t^T|\Psi(s,x)|ds w(x)dx \leq
\Int_{\R^d}\int_t^T\E(|\Psi(s,X_{t,s}(x))|)ds w(x)dx\leq
C\Int_{\R^d}\Int_t^T|\Psi(s,x)|ds w(x)dx.
\end{equation}
 \end{Proposition}
 This  equivalence of norms result  plays also  an important role in the proof of the existence of the solution for SPDE as a connection between the functional norms for such solutions and random norms from  BDSDE solutions.
 
Under regular conditions (Assumption \ref{assdiff}) on the diffusion, it is known that the stochastic flow  solution of a continuous SDE satisfies the homeomorphic property (see Bismut \cite{b}, Kunita \cite{K84}, \cite{K90}). We have  the following result whose proof can be found in  \cite{K84}.
\begin{Proposition}\label{flow}
Let Assumption \ref{assdiff} holds. Then
$\{X_{t,s}(x); x \in \mathbb{R}^d \}$ is a $C^2$-diffeomorphism a.s.
stochastic flow. Moreover the inverse of the flow satisfies the
following backward SDE
\begin{equation}\label{inverse:flow}
\begin{split}
X_{t,s}^{-1}(y) &  = y - \int_t^s \widehat{b}(X_{r,s}^{-1}(y)) dr  -
\int_t^s \sigma (X_{r,s}^{-1} (y)) d\B_r .
\end{split}
\end{equation}
for any  $t<s$,  where
\begin{equation}
\label{drift:backward}
\begin{split}
\widehat{b}(x) =  b(x) -  \sum_{i,j}\frac{\partial \sigma^j (x)
}{\partial x_i}  \sigma^{ij} (x) .
\end{split}
\end{equation}
\end{Proposition}
\vspace{0.2cm}
\noindent We denote by $J(X_{t,s}^{-1}(x))$ the determinant of the Jacobian
matrix of $X_{t,s}^{-1}(x)$, which is positive and
$J(X_{t,t}^{-1}(x))=1$. For $\varphi\in C_c^{\infty}(\R^d)$, we define
a process $\varphi_t:\,\Omega\times [t,T]\times \R^d\rightarrow \R$ by
\begin{equation}
\label{random:testfunction conv}
\varphi_t(s,x):=\varphi(X_{t,s}^{-1}(x))J(X_{t,s}^{-1}(x)).
\end{equation}

By a change of variable formula, we have for  all $v\in \mathbf{L}^2(\R^d)$,   $$(v\circ X_{t,s} (\cdot),
\varphi)=\Int_{\R^d}v(X_{t,s}
(x))\varphi(x)dx=\Int_{\R^d}v(y)\varphi(X_{t,s}^{-1}(y))J(X_{t,s}^{-1}(y))dy
=(v,\varphi_t(s,\cdot)), \; \, \P-a.s.$$ 
Since  $ (\varphi_t(s,x))_{ t\leq s}$ is a
process,  we may not use it directly as a test function because 
$\Int_t^T(u(s,\cdot),\partial_s\varphi_t(s,\cdot)){\color{blue}{ds}}$ has no sense. However
$\varphi_t(s,x)$ is a semimartingale and we have the following
decomposition of $\varphi_t(s,x)$ where the proof can be found  in \cite{BM01} (proof of Lemma 2.1. p.135), see also Kunita \cite{K94a}, \cite{K94b} for the use of such random test functions.
\begin{Lemma}\label{decomposition conv}
For every function $\varphi\in C_c^{\infty}(\R^d),$
\begin{equation}\label{decomp conv}\begin{array}{ll}
\varphi_t(s,x)&=\varphi(x)+\displaystyle\Int_t^s{\mathcal
L}^\ast\varphi_t(r,x)dr-\sum_{j=1}^{d}\Int_t^s\left(\sum_{i=1}^{d}\frac{\partial}{\partial
x_i}(\sigma^{ij}(x)\varphi_t(r,x))\right)dB_r^j,
 \end{array} \end{equation}
where ${\mathcal L}^\ast$ is the adjoint operator of ${\mathcal L}$.
\end{Lemma}
\noindent Thanks to the above lemma, we can replace $\partial_s \varphi\, ds$ by the It\^o stochastic integral with respect to $d\varphi_t(s,x)$. This allows us to give the following
\begin{Definition}
For every $s\in[t,T]$, $u\in\Hc_T$ and $ \varphi\in C_c^{\infty}(\R^d),$ we define 
\begin{equation}\label{decomp conv}\begin{array}{ll}
\Int_s^T \left(u_r,d\varphi_t(r,.)\right)&=\displaystyle\Int_s^T \left(u_r,{\mathcal
L}^\ast\varphi_t(r,.)\right)dr-\sum_{j=1}^{d}\Int_s^T\left(\sum_{i=1}^{d}\left(u_r,\frac{\partial}{\partial x_i}(\sigma^{ij}(x)\varphi_t(r,.))\right)\right)dB_r^j,
 \end{array} \end{equation}
\end{Definition} 
\subsection{Existence and uniqueness of solutions for the reflected SPDE}
\label{subsection:SPDE}
In order  to provide a probabilistic representation to the solution of the RSPDEs \eqref{OSPDE1}, we introduce the following  Markovian RBDSDE:
\begin{equation}
\label{rbsde1}
 \left\lbrace
\begin{aligned}
&(i)~ Y_{s}^{t,x}
 =\Phi(X_{t,T}(x))+
\Int_{s}^{T}f_r(X_{t,r}(x),Y_{r}^{t,x},Z_{r}^{t,x})dr+\Int_{s}^{T}h_r(X_{t,r}(x),Y_{r}^{t,x},Z_{r}^{t,x})d\W_r+K_{T}^{t,x}-K_{s}^{t,x}\\
&\hspace{2.5cm}
 -\Int_{s}^{T}Z_{r}^{t,x}dB_{r},\;
\P\text{-}a.s. , \; \forall \,  s \in [t,T]  \\
& (ii)~ Y_{s}^{t,x} \in \bar{D} \, \, \quad \P\text{-}a.s.\\
& (iii) \Int_0^T (Y_{s}^{t,x}-v_s(X_{t,s}(x)))^* dK_{s}^{t,x}\leq 0., \, \, \P\text{-}a.s., \\
&~ \text{for any continuous }\, \Fc_t -\text{random function}\, v \, : \,[0,T] \times \Omega \times \mathbb R^d \longrightarrow \,  \bar{D}.
\end{aligned}
\right.
\end{equation}

\noindent Moreover, using Assumptions \ref{assxi} and \ref{assgener} and  the equivalence of norm results  (\ref{equi1 conv}) and (\ref{equi2 conv}), we get
\b*
\begin{split}
&\Phi(X_{t,T}(x)) \in \mathbf{L}^2({\cal F}_T) \, \,  \mbox{and} \, \,    \Phi(X_{t,T}(x)) \in \bar{D},\\
 &f_s^0(X_{t,s}(x)) \in \mathcal {H}_{k}^2(t,T)\,\mbox{ and }\,h_s^0(X_{t,s}(x)) \in \mathcal {H}_{k\times d}^2(t,T).\\
\end{split}
\e*
Therefore  under Assumption \ref{assxi}-\ref{assint} and according to Theorem \ref{existence:RBDSDE},  there exists a unique triplet 
$ (Y^{t,x},Z^{t,x},K^{t,x}) $ solution of the RBDSDE \eqref{rbsde1} associated to $ (\Phi, f, h)$.\\

\noindent We  now consider  the  following definition  of weak solutions for the reflected SPDE (\ref{OSPDE1}):
\begin{Definition}[{\textbf{Weak solution of RSPDE}}]
\label{o-pde}  We say that $(u,\nu ):= (u^i,\nu^i )_{1\leq i\leq k}$ is the weak solution of the reflected SPDE (\ref{OSPDE1}) associated to $(\Phi,f,h)$, if for each $1\leq i\leq k$ 
\begin{itemize}
\item[(i)] $ u\in\mathcal H_T$, $u_t(x)\in \bar{D}, dx\otimes dt\otimes d\P~a.e.$, and  $u(T,x)=\Phi(x)$.
\item[(ii)] $\nu^i $ is a signed \textit{Random measure} on $[0,T]\times\R^d$  concentrated  on $ \{u\in \partial D\} $ such that:
\begin{itemize}
\item[a)] $\nu^i $ is adapted in the sense that for any measurable and $|\nu^i|$-integrable function $\psi:[0,T]\times \R^d\longrightarrow\R^d$ and for each $s\in[t,T]$,$\Int_{s}^{T}\!\Int_{\mathbb{R}^{d}}\!\psi (r,x)\nu^i(dr,dx)$ is $\Fc_{s,T}^W$-measurable.
\item[b)] $\E\big[\Int_{0}^{T}\Int_{\mathbb{R}^{d}} w (x)|\nu^i| (dt,dx)\big]<\infty.$

\end{itemize} 
\item[(iii)] For every $\varphi \in \mathcal D_T$, 
\begin{align}\label{OPDE}
\nonumber &\Int_{t}^{T}\!\!\Int_{\mathbb{R}^{d}}\!u^i(s,x)\partial _{s}\varphi(s,x)dxds+\Int_{\mathbb{R}^{d}}\!\!(u^i(t,x )\varphi (t,x
)-\Phi^i(x )\varphi (T,x))dx-\Int_{t}^{T}\!\!\Int_{\mathbb{R}^{d}}\!u^i(s,x)\mathcal{L}^*\varphi(s,x)dxds\\
\nonumber &=\Int_{t}^{T}\!\!\Int_{\mathbb{R}^{d}}\!\!f_s(x ,u(s,x),\nabla u(s,x)\sigma(x))\varphi(s,x)dxds+\Int_{t}^{T}\!\!\Int_{\mathbb{R}^{d}}\!\!h_s(x ,u(s,x),\nabla u(s,x)\sigma(x))\varphi(s,x)dxd\W_s\\
& +\Int_{t}^{T}\!\!\Int_{\mathbb{R}^{d}}\!\!\varphi (s,x)1_{\{u\in \partial D\}}(s,x)\nu^i(ds,dx), \; \, \P-a.s. 
\end{align}
\end{itemize}
For the sake of simplicity we will omit in the sequel the subscript $i$.
\end{Definition}
\vspace{0.2cm}
\noindent The main result of this section is the following:
\begin{Theorem}
\label{existence:RSPDE}
Let  Assumptions \ref{assumptiondomain} and \ref{assxi}-\ref{assint}   hold and  $w (x)=(1+\left| x\right|)^{-p}$ with $p >  d+1 $. Then
there exists a weak solution $(u,\nu
)$ of  the reflected SPDE (\ref{OSPDE1}) associated to $(\Phi,f,h)$ such that, $ u (t, x) := Y_t^{t,x}$,   $dt\otimes
d\P\otimes w(x)dx-a.e.$,  and  \be\label{con-pre}
 Y_{s}^{t,x}=u(s,X_{t,s}(x)), \quad \quad Z_{s}^{t,x}=(\nabla
u\sigma)(s,X_{t,s}(x)), \quad ds\otimes
d\P\otimes w(x)dx-a.e.. 
\ee
Moreover, $\nu^i $ is a \textit{regular measure} in the following sense: for every measurable bounded and positive functions $\varphi $ and $\psi $,
\begin{align}
\nonumber &\Int_{\mathbb{R}^{d}}\Int_{t}^{T}\varphi (s,X^{-1}_{t,s}(x))J(X^{-1}
_{t,s}(x))\psi (s,x)1_{\{u\in \partial D\}}(s,x)\nu^i (ds,dx)\\
&=\Int_{\mathbb{R}^{d}}\Int_{t}^{T}\varphi (s,x)\psi (s,X_{t,s}(x))dK_{s}^{t,x,i}dx\text{, a.s.}
\label{con-k}
\end{align}
where $(Y_{s}^{t,x},Z_{s}^{t,x},K_{s}^{t,x})_{t\leq s\leq T}$ is the
solution of RBDSDE (\ref{rbsde1}) and satisfying the probabilistic interpretation \eqref{con-k}.\\
If $(\overline{u},\overline{\nu })$ is another solution of the reflected SPDE (\ref{OSPDE1}) such that $\overline{\nu }$ satisfies (\ref{con-k}) with some $%
\overline{K}$ instead of $K$, where $(\overline{K}_s^{t,x})_{t\leq s\leq T}$ is a continuous finite variation process for all $(t,x)$, then $\overline{u}=u$ and $\overline{\nu }=\nu $.\\
In other words, there is a unique Randon regular measure with support $\{u\in\partial D\}$ which satisfies (\ref{con-k}).
\end{Theorem}
\begin{Remark}
The expression (\ref{con-k}) gives us the probabilistic
interpretation (Feymamn-Kac's formula) for the measure $\nu $ via
the nondecreasing process $K^{t,x}$ of the RBDSDE. This formula was
first introduced in Bally et al. \cite{BCEF} (see also \cite{MX08}) in the context of obstacle problem  for PDEs. Here we adapt this notion to the case of SPDEs in a convex domain.
\end{Remark}
\vspace{0.2cm}
\noindent We give now the following result which allows us to link in a natural way the solution of RSPDE with the associated RBDSDE.
 Roughly speaking,  for each  test function  $ \varphi \in C_c^{\infty}(\R^d) $,  the variational formulation \eqref{OPDE}  written with  the random test functions $\varphi_t(\cdot,\cdot)$  gives the RBDSDE \label{rbsde1}  integrated against the test function $ \varphi $ , this dictionary can  be understood as the dual formulation of the variational equation for SPDE.  Pardoux and Peng  \cite{pp1994} have proved the probabilistic representation of classical solution $u$  for semilinear SPDE's  via BDSDEs by using the classical Itô's formula  for  $u (s, X_{t,s} (x))$.  However, since we consider Sobolev solutions for RSPDEs, the following proposition plays the rule of Itô's formula applied to the random test function $\varphi_t(s,x)$.  
    \begin{Proposition}
\label{weak:Itoformula1 conv} Let  Assumptions \ref{assxi}-\ref{assint} hold and $(u,\nu)$ be a weak solution of the reflected 
SPDE (\ref{OSPDE1}) associated to $(\Phi,f,h)$, 
then for $s\in[t,T]$ and $\varphi\in
C_c^{\infty}(\R^d)$, 
\begin{equation}\label{wspde2}
\begin{array}{ll}
\displaystyle\int_{\R^d}\int_s^Tu(r,x)d\varphi_t(r,x)dx+(u(s,\cdot),\varphi_t(s,\cdot))-(\Phi(\cdot),\varphi_t(T,\cdot))
-\Int_{\R^d}\int_s^T u(r,x)\mathcal L^\ast \varphi_t(r,x))drdx\\
=\displaystyle\int_{\R^d}\int_s^Tf_r(x,u(r,x),\nabla
u(r,x)\sigma(x))\varphi_t(r,x)drdx+\displaystyle\int_{\R^d}\int_s^T h_r(x,u(r,x),\nabla
u(r,x)\sigma(x))\varphi_t(r,x)d\W_rdx\\
+\displaystyle\int_{\mathbb{R}^{d}}\int_{s}^{T}\varphi _{t}(r,x)1_{\{u\in\partial D\}}(r,x)\nu (dr,dx) \quad a.s.
\end{array}
\end{equation}
where  $\Int_{\R^d}\Int_s^Tu(r,x)d\varphi_t(r,x)dx $ is well defined in the semimartingale decomposition result (Lemma \ref{decomposition conv}).
\end{Proposition}
The proof of the proposition is  the same as the proof of Proposition 2.3, p. 137 in Bally and Matoussi  \cite{BM01}.  This latter is based on Lemma 4.1 p.147 and Lemma 4.2. p.148  which involve the Wong-Zakai approximation of the Itô stochastic integral appearing in the semimartingale decomposition of the random test functions  given by \eqref{random:testfunction conv}  (see \cite{iw}, chap. 6, section 7 , p.480-517). The main idea is to use  $(\varphi_t(s,x)) $ as a test function in the  \eqref{OPDE}. The problem is that $(\varphi_t(s,x)) $  is not differentiable in the time variable $s$,  so that   $ \int \int_t^T u_s \partial_s \varphi_t  (s,x) ds dx $ has no sense. However     $(\varphi_t(s,x)) $ a semimartingale and one can use  Wong-Zakai approximation (see \cite{BM01},  Lemma 4.2 p.148) to handel with this point . Of course, the WZ approximation converges to Stratonovitch stochastic integral, but thanks to regularity assumption on the diffusion coefficient $ \sigma (x)$, one can write  the result explicitly as Itô's stochastic integral with drift term which disappear with a part of the drift term in the semimartingale decomposition of  $ (X^{-1}_{t,s}(x)) $.  
 We note that  Kunita  (\cite{K94a}, \cite{K94b}) has developed a theory of distribution valued semimartingales to study a class of linear SPDEs using these random test functions.  

\vspace{0.5em}
\noindent\textbf{Proof of Theorem \ref{existence:RSPDE}.}\\
\noindent\textbf{a) Existence}: The existence of a solution
will be proved in two steps. For the first step, we suppose that
$h$ doesn't depend on $y,z$, then we are able to apply the classical penalization method. 
In the second step, we study the case when $h$ depends on $y,z$ with the result obtained in the first step.\\[0.1cm] 

\textit{Step 1} :
We will use the penalization method. For $n\in \mathbb{N}$, we consider for
all $s\in [t,T]$,
\begin{align*}
Y_{s}^{n,t,x}=\Phi(X_{t,T}(x))&+\int_{s}^{T}f_r(X_{t,r}(x),Y_{r}^{n,t,x},Z_{r}^{n,t,x})dr+\int_{s}^{T}h_r(X_{t,r}(x))d\W_r\\
&-n
\int_{s}^{T}(Y_{r}^{n,t,x}-\pi(Y_{r}^{n,t,x}))dr
-\int_{s}^{T}Z_{r}^{n,t,x}dB_{r}.
\end{align*}

\noindent From Theorem 3.1 in Bally and Matoussi \cite{BM01}, we know that $u_{n}(t,x):=Y_{t}^{n,t,x}$, is a solution of the SPDE $(\Phi,f_{n},h)$ (\ref{SPDE1}), with $
f_{n}(t,x,y)=f(t,x,y,z)-n(y-\pi(y))$, i.e. for every $\varphi \in\mathcal{D}_T$
\begin{align}\label{o-equa1}
 \nonumber\Int_{t}^{T}(u^{n}(s,\cdot),\partial
_{s}\varphi(s,\cdot) )ds & +(u^{n}(t,\cdot ),\varphi
(t,\cdot ))-(\Phi(\cdot ),\varphi (T,\cdot))-\Int_{t}^{T}
(u^{n}(s,\cdot),\mathcal{L}^*\varphi(s,\cdot))ds\\
\nonumber&=\Int_{t}^{T}(f_s(\cdot,u^{n}(s,\cdot),\sigma ^{*}\nabla
u_{n}(s,\cdot)),\varphi(s,\cdot))ds+\Int_{t}^{T}(h_s(\cdot),\varphi(s,\cdot))d\W_s\\
&-n\int_{t}^{T}((u^{n}-\pi(u^{n}))(s,\cdot ),\varphi(s,\cdot))ds.
\end{align}
Moreover from Theorem 3.1 in Bally and Matoussi \cite{BM01}, we also have
\begin{align}\label{rep1}
\nonumber &Y_{s}^{n,t,x}=u_{n}(s,X_{t,s}(x))\,\,\,,\,\,\,Z_{s}^{n,t,x}=(\nabla
u_{n}\sigma)(s,X_{t,s}(x)),\,\, ds\otimes
d\P\otimes w(x)\,dx-a.e. \\
\end{align}
Set $K_{s}^{n,t,x}=-n \displaystyle
\Int_{t}^{s}(Y_{r}^{n,t,x}-\pi(Y_{r}^{n,t,x}))dr$. Then by
(\ref{rep1}), we have that 
\be
K_{s}^{n,t,x}=-n \displaystyle
\Int_{t}^{s}(u_{n}-\pi(u_{n}))(r,X_{t,r}(x))dr.\label{defKn}\ee
Following the estimates and convergence results for
$(Y^{n,t,x},Z^{n,t,x},K^{n,t,x})$ in Section 2 and  estimate (\ref{estunif}),  we get :
\begin{eqnarray}
\label{K-estimate}
 &\underset{n}{\Sup}\, \E\left[\underset{t\leq s\leq T}{\Sup}\left| Y_{s}^{n,t,x}\right| ^{2}+\int_{t}^{T}\left\|
Z_{s}^{n,t,x}\right\|^2ds + \|K^{n,t,x}\|_{VT}\right] \leq C\left( T,x\right),
\end{eqnarray} 
where  $$ C (T,x):=  {\color{blue}{C}}\,\E \Big[ \left| \Phi (X_{t,T} (x))\right|^2 + \Int_t^T  \big( \left| f_s^0 (X_{t,s} (x))\right|^2 \,  + \left| h_s^0 (X_{t,s} (x))\right|^2 \, \big)\, ds \Big],$$
and
\begin{align*}
& \E[\underset{t\leq s\leq T}{\Sup}\left| Y_{s}^{n,t,x}-Y_{s}^{m,t,x}\right|
^{2}]+\E[\Int_{t}^{T}\left\| Z_{s}^{n,t,x}-Z_{s}^{m,t,x}\right\|
^{2}ds]\\
&+\E[\underset{t\leq s\leq T}{\Sup}\left|
K_{s}^{n,t,x}-K_{s}^{m,t,x}\right| ^{2}]\quad\longrightarrow 0,~\text{as}~ n,m \longrightarrow +\infty.
\end{align*}
Moreover, the equivalence of norms results (\ref{equi2 conv}) yield:
\begin{eqnarray*}
&&\Int_{\mathbb{R}^{d}}\Int_{t}^{T} w(x)(\left|
u_{n}(s,x)-u_{m}(s,x)\right| ^{2}+\left| (\nabla u_{n}\sigma)(s,x)-(\nabla u_{m}\sigma)(s,x)\right| ^{2})dsdx \\
&\leq &\frac{1}{k_{2}}\Int_{\mathbb{R}^{d}} w(x)\E\Int_{t}^{T}(\left|
Y_{s}^{n,t,x}-Y_{s}^{m,t,x}\right| ^{2}+\left\|
Z_{s}^{n,t,x}-Z_{s}^{m,t,x}\right\| ^{2})dsdx\longrightarrow 0.
\end{eqnarray*}
Thus $(u_{n})_{n\in\mathbb N}$ is a Cauchy sequence in $\mathcal{H}_T$, and the limit $%
u=\underset{n\rightarrow \infty }{\lim}u_{n}$ belongs to  $\mathcal{H}_T$.\\
Moreover, using standard computations from BSDEs and SPDEs technics we can prove that $\E[\underset{s\in[0,T]}{\sup} \|u_s^n-u_s^m\|_2^2]\longrightarrow 0 \,\text{as}\, n,m\rightarrow \infty$.\\
Denote $\nu _{n}(dt,dx)= -n(u_{n}-\pi(u_{n}))(t,x)dtdx$ and $\pi _{n}(dt,dx)=w
(x)\nu _{n}(dt,dx)$, then by the equivalence norm result (\ref{equi2 conv}) we get 
\begin{eqnarray*}
\E\big[|\pi _{n}|([t,T]\times \mathbb{R}^{d})\big] &=&\Int_{\mathbb{R}^{d}}\int_{t}^{T}\E\big[ n|(u_{n}-\pi(u_{n}))(s,x)| \big]ds w
(x)dx \\
&\leq &\frac{1}{k_{2}}\Int_{\mathbb{R}^{d}}\int_{t}^{T}\E\big[ n|(u_{n}-\pi(u_{n}))(s,X_{t,s}(x))| \big]ds w
(x)dx.
\end{eqnarray*}
Finally, using \eqref{defKn} and \eqref{K-estimate}, we obtain
\begin{eqnarray*}
\E\big[|\pi _{n}|([0,T]\times \mathbb{R}^{d})\big]&\leq &\frac{1}{k_{2}}\Int_{\mathbb{R}^{d}} w(x)\E\left\|
K^{n,0,x}\right\|_{VT} dx\leq C\Int_{\mathbb{R}^{d}} w(x)dx<\infty .
\end{eqnarray*}
It follows that
\begin{equation}
\underset{n}{\Sup}\, \E\big[|\pi _{n}|([0,T]\times \mathbb{R}^{d})\big]<\infty .  \label{est-measure}
\end{equation}
Moreover  by Lemma \ref{tight} (see Appendix \ref{appendix:tight}),  the sequence of measures $(\pi _{n}(\omega,ds,dx))_{n  \in \mathbb N}$  is tight
$\P$-a.s. in $\omega\in\Omega$.  Therefore, there exits a subsequence such that $(\pi _{n}(\omega,ds,dx))_{n  \in \mathbb N}$  converges  weakly to a measure $\pi (\omega,ds,dx)$. Define $\nu:=w^{-1}\pi$, it remains now to prove that the limit $\nu (\omega,ds,dx)$ is measurable with respect to $\omega$ and satisfying Definition \ref{o-pde}-(ii).  We note first  that from \eqref{est-measure} and Fatou's Lemma, we get $\E\big[|\pi|([0,T]\times \mathbb{R}^{d})\big]<\infty$. We have also for $\varphi \in \mathcal{D}_T$ with compact support in $x$,%
$$
\Int_{\mathbb{R}^{d}}\Int_{t}^{T}\varphi d\nu _{n}=\Int_{\mathbb{R}%
^{d}}\Int_{t}^{T}\frac{\phi }{w }d\pi _{n}\rightarrow \Int_{\mathbb{R}
^{d}}\Int_{t}^{T}\frac{\phi }{w }d\pi =\Int_{\mathbb{R}
^{d}}\Int_{t}^{T}\varphi d\nu .
$$
Passing to the limit in the SPDE $(\Phi,f_{n},h)$ (\ref{o-equa1}), we get that $(u,\nu )$
satisfies the following equation, i.e. for every $\varphi \in
\mathcal{D}_T$, we have
\begin{eqnarray}
&&\Int_{t}^{T}(u(s,\cdot),\partial _{s}\varphi(s,\cdot) )ds+(u(t,\cdot ),\varphi (t,\cdot))-(\Phi(\cdot),\varphi (T,\cdot))-\Int_{t}^{T}(u(s,\cdot),\mathcal{L}^*\varphi(s,\cdot))ds
\nonumber \\
&-&\Int_{t}^{T}(f_s(\cdot, u(s,\cdot),\sigma^{*}\nabla u(s,\cdot)),\varphi(s,\cdot) )ds-\Int_{t}^{T}(h_s(\cdot),\varphi(s,\cdot) )d\W_s=\Int_{t}^{T}\Int_{\mathbb{R}
^{d}}\varphi (s,x)\nu (ds,dx),\, \P-a.s.\nonumber\\
  \label{equa1} 
\end{eqnarray}
Thus, the term in the right hand side of the last equation is measuble with respect to $\omega$ since it is the sum of measurable terms obtained as the limit of sequence of measurable processes. Therefore, the pair $(u,\nu )$ satisfies the reflected SPDE associated to  $(\Phi,f,h)$.\\
The last point is to prove that $\nu $ satisfies the probabilistic interpretation (\ref{con-k}). Since $K^{n,t,x}$ converges to $K^{t,x}$ uniformly in $t$, the
measure $dK^{n,t,x}$ converges to $dK^{t,x}$ weakly in probability.
Fix two continuous functions $\varphi $, $\psi $ : $[0,T]\times \mathbb{R}%
^{d}\rightarrow \mathbb{R}^{+}$ which have compact support in $x$ and a
continuous function with compact support $\theta :\mathbb{R}^{d}\rightarrow %
\mathbb{R}^{+}$, from Bally et al \cite{BCEF} (The proof of Theorem 4), we have (see also Matoussi and Xu \cite{MX08})
\begin{eqnarray*}
&&\int_{\mathbb{R}^{d}}\int_{t}^{T}\varphi (s,X^{-1}_{t,s}(x))J(X
^{-1}_{t,s}(x))\psi (s,x)\theta (x)\nu (ds,dx) \\
&=&\lim_{n\rightarrow \infty }-\int_{\mathbb{R}^{d}}\int_{t}^{T}\varphi (s,
X^{-1}_{t,s}(x))J(X^{-1}_{t,s}(x))\psi (s,x)\theta
(x)n(u_{n}-\pi(u_{n}))(s,x)dsdx \\
&=&\lim_{n\rightarrow \infty }-\int_{\mathbb{R}^{d}}\int_{t}^{T}\varphi
(s,x)\psi (s,X_{t,s}(x))\theta
(X_{t,s}(x))n(u_{n}-\pi(u_{n}))(t,X_{t,s}(x))dtdx \\
&=&\lim_{n\rightarrow \infty }\int_{\mathbb{R}^{d}}\int_{t}^{T}\varphi
(s,x)\psi (s,X_{t,s}(x))\theta (X_{t,s}(x))dK_{s}^{n,t,x}dx \\
&=&\int_{\mathbb{R}^{d}}\int_{t}^{T}\varphi (s,x)\psi (s,X_{t,s}(x))\theta
(X_{t,s}(x))dK_{s}^{t,x}dx.
\end{eqnarray*}
We take $\theta =\theta _{R}$ to be the regularization of the indicator
function of the ball of radius $R$ and pass to the limit with $R\rightarrow
\infty $, to get that
\begin{equation}\label{con-k1}
\Int_{\mathbb{R}^{d}}\int_{t}^{T}\varphi (s,X^{-1}_{t,s}(x))J(X^{-1}
_{t,s}(x))\psi (s,x)\nu (ds,dx)=\int_{\mathbb{R}^{d}}\int_{t}^{T}\varphi
(s,x)\psi (s,X_{t,s}(x))dK_{s}^{t,x}dx.
\end{equation}
From Section 2, it follows that $dK_{s}^{t,x}=1_{\{u\in\partial D\}}(s,X_{t,s}(x))dK_{s}^{t,x}$. Again by regularization procedure, we can set $\psi =1_{\{u\in\partial D\}}$ in (\ref{con-k1}) to obtain
\begin{align*}
&\Int_{\mathbb{R}^{d}}\int_{t}^{T}\varphi (s,X^{-1}_{t,s}(x))J(X^{-1}
_{t,s}(x))1_{\{u\in\partial D\}}(s,x)\nu (ds,dx)\\
&=\Int_{\mathbb{R}^{d}}\int_{t}^{T}\varphi
(s,X^{-1}_{t,s}(x))J(X^{-1}_{t,s}(x))\nu (ds,dx)\text{, a.s.}
\end{align*}
Note that the family of functions $A(\omega )=\{(s,x)\rightarrow \varphi (s,%
X^{-1}_{t,s}(x)):\varphi \in C_{c}^{\infty }\}$ is an algebra which
separates the points (because $x\rightarrow X^{-1}_{t,s}(x)$ is a
bijection). Given a compact set $G$, $A(\omega )$ is dense in $C([0,T]\times
G)$. It follows that $J(X^{-1}_{t,s}(x))1_{\{u\in\partial D\}}(s,x)\nu (ds,dx)=J(
X^{-1}_{t,s}(x))\nu (ds,dx)$ for almost every $\omega $. While $J(
X^{-1}_{t,s}(x))>0$ for almost every $\omega $, we get $\nu
(ds,dx)=1_{\{u\in\partial D\}}(s,x)\nu (ds,dx)$, and (\ref{con-k}) follows.\\
Then we get easily that $Y_{s}^{t,x}=u(s,X_{t,s}(x))$ and $
Z_{s}^{t,x}=(\nabla u\sigma)(s,X_{t,s}(x))$,  in view of the convergence
results for $(Y_{s}^{n,t,x},Z_{s}^{n,t,x})$ and the equivalence of norms. So $u(s,X_{t,s}(x))=Y_{s}^{t,x}\in \bar{D}$. Specially for $s=t$, we
have $u(t,x)\in \bar{D}$.\\[0.2cm]
\textit{ Step 2 } : \textit{The  nonlinear  case where $h$
depends on $y$ and $z$}.\\  
Let $(Y_s^{t,x},Z_s^{t,x},K_s^{t,x})$ the unique solution of RBDSDE \eqref{RBDSDE} associated to $(\Phi,f,h)$. In order to avoid standard fixed point arguments, we define $H(s,x)\triangleq h(x,X_{t,s}(x),Y_s^{t,x},Z_s^{t,x}).$ Due to the fact that
$h$ is Lipschitz with respect to $(y,z)$, we have
$$ |H(s,x)|\leq |h_s^0(X_{t,s}(x))|+ C( |Y_s^{t,x}|+ |Z_s^{t,x}|).$$
Besides, by standard computations similar to estimate \eqref{K-estimate} we get
\begin{eqnarray}\label{estimate}
 &\E\left[\left| Y_{s}^{t,x}\right| ^{2}+\Int_{t}^{T}\left\|
Z_{r}^{t,x}\right\|^2dr + \|K^{t,x}\|_{VT}\right] \leq C\left( T,x\right),
\end{eqnarray} 
where  $$ C (T,x):=  C\,\E \Big[ \left| \Phi (X_{t,T} (x))\right|^2 + \Int_t^T  \big( \left| f_s^0 (X_{t,s} (x))\right|^2 \,  + \left| h_s^0 (X_{t,s} (x))\right|^2 \, \big)\, ds \Big].$$
Integrating the estimate \eqref{estimate} over $x$ with respect to the measure $w(x)dx$ and thanks to the equivalence norm results, we conclude that $H$ belongs to $ L^2([0,T]\times\Omega\times \R^d; dt\otimes d\P\otimes w(x)dx)$. Then, 
applying the result of Step 1 yields that there exists $(u,\nu )$ where 
$u(t,x):=Y_t^{t,x}$,  $dt\otimes d\P\otimes w(x)dx-a.e.$ and satisfying the SPDE with obstacle $(\Phi,f,H)$, i.e. for every $\varphi \in
\mathcal{D}_T$, we have
\begin{eqnarray}
&&\Int_{t}^{T}(u(s,\cdot),\partial _{s}\varphi(s,\cdot) )ds+(u(t,\cdot ),\varphi (t,\cdot))-(\Phi(\cdot),\varphi (T,\cdot))-\Int_{t}^{T}(u(s,\cdot),\mathcal{L}^*\varphi(s,\cdot))ds
\nonumber \\
&=&\Int_{t}^{T}(f_s(\cdot, u(s,\cdot),\sigma^{*}\nabla u(s,\cdot)),\varphi(s,\cdot) )ds+\Int_{t}^{T}(H_s(\cdot),\varphi(s,\cdot) )d\W_s\nonumber\\
&+&\int_{t}^{T}\Int_{\mathbb{R}
^{d}}\phi (s,x)1_{\{u\in\partial D\}}(s,x)\nu (ds,dx).  \label{equa2}
\end{eqnarray}
Then by the uniqueness of the solution to the RBDSDE ($\Phi(X_{t,T}(x))$, $f$, $h$),  we get easily that $Y_{s}^{t,x}=u(s,X_{t,s}(x))$, $%
Z_{s}^{t,x}=(\nabla u\sigma)(s,X_{t,s}(x))$, and $\nu$ satisfies the probabilistic interpretation (\ref{con-k}). So $u(s,X_{t,s}(x))=Y_{s}^{t,x}\in\bar{D}$. Specially for $s=t$, we have $u(t,x)\in\bar{D}$, which is the desired result.\\

\noindent\textbf{b) Uniqueness } :  Set $(\overline{u},\overline{\nu
})$ to be another weak solution of the reflected SPDE (\ref{OSPDE1}) associated to 
$(\Phi,f,h)$; with $\overline{\nu }$ verifies (\ref{con-k}) for a continuous finite variation process $(\overline{K}_s^{t,x})_{t\leq s\leq T}$. We fix $\varphi :\mathbb{R}^{d}\rightarrow \mathbb{R}^k$, a smooth function in $C_{c}^{2}(%
\mathbb{R}^{d})$ with compact support and denote $\varphi _{t}(s,x)=\varphi (X^{-1}_{t,s}(x))J(X^{-1}_{t,s}(x))$. From Proposition \ref{weak:Itoformula1 conv}, one may use $\varphi _{t}(s,x)$ as a test function in the SPDE $(\Phi,f,h)$ with $\partial _{s}\varphi (s,x)ds$ replaced by a stochastic integral with respect to the semimartingale $\varphi _{t}(s,x)$. Then we get, for $t\leq s\leq T$

\begin{align}
&\Int_{\mathbb{R}^{d}}\int_{s}^{T}\overline{u}(r,x)d\varphi _{t}(r,x)dx+\Int_{\mathbb{R}^{d}}\overline{u}(s,x)\varphi _{t}(s,x)dx-\Int_{\mathbb{R}^{d}}\Phi(x)\varphi _{t}(T,x)dx-\Int_{s}^{T}\int_{\mathbb{R}^{d}}\overline{u}(r,x)\mathcal{L}^*\varphi _{t}(r,x)drdx  \label{o-pde-u1}
\nonumber \\
&=\Int_{s}^{T}\Int_{\mathbb{R}^{d}}f_r(x,\overline{u}(r,x),(\nabla \overline{u}\sigma)(r,x))\varphi _{t}(r,x)drdx+ \Int_{s}^{T}\Int_{\mathbb{R}^{d}}h_r(x,\overline{u}(r,x),(\nabla \overline{u}\sigma)(r,x))\varphi _{t}(r,x)dx d\W_r\nonumber \\
& +\Int_{s}^{T}\Int_{\mathbb{
R}^{d}}\varphi _{t}(r,x)1_{\{\overline{u}\in\partial D\}}(r,x)\overline{\nu }(dr,dx).
\end{align}
By (\ref{decomp conv}) in Lemma \ref{decomposition conv}, we have

\begin{align*}
&\Int_{\mathbb{R}^{d}}\Int_{s}^{T}\overline{u}(r,x)d\varphi _{t}(r,x)dx
=\Int_{s}^{T}(\Int_{\mathbb{R}^{d}}(\nabla \overline{u}\sigma)(r,x)\varphi_{t}(r,x)dx)dB_{r}+\Int_{s}^{T}\Int_{\mathbb{R}^{d}}\overline{u}(r,x)\mathcal{L}^*\varphi _{t}(r,x)drdx .\\
\end{align*}
Substituting this equality in (\ref{o-pde-u1}), we get
\begin{align*}
&\Int_{\mathbb{R}^{d}}\overline{u}(s,x)\varphi _{t}(s,x)dx =\Int_{\mathbb{R}^{d}}\Phi(x)\varphi _{t}(T,x)dx-\Int_{s}^{T}(\Int_{\mathbb{R}^{d}}(\nabla \overline{u}\sigma)(r,x)\varphi
_{t}(r,x)dx)dB_{r}\\
&+\Int_{s}^{T}\Int_{\mathbb{R}^{d}}f_r(x,\overline{u}(r,x),(\nabla \overline{u}\sigma)(r,x))\varphi _{t}(r,x)drdx+ \Int_{s}^{T}\Int_{\mathbb{R}^{d}}h_r(x,\overline{u}(r,x),(\nabla \overline{u}\sigma)(r,x))\varphi _{t}(r,x)dx d\W_r\\&+\Int_{s}^{T}\Int_{\mathbb{
R}^{d}}\varphi _{t}(r,x)1_{\{\overline{u}\in\partial D\}}(r,x)\overline{\nu }(dr,dx).
\end{align*}
Then by changing of variable $y=X^{-1}_{t,r}(x)$ and applying (\ref
{con-k}) for $\overline{\nu }$, we obtain
\begin{align*}
&\Int_{\mathbb{R}^{d}}\overline{u}(s,X_{t,s}(y))\varphi (y)dy =\Int_{\mathbb{R}^{d}}\Phi(X_{t,T}(y))\varphi (y)dy\\
&+\Int_{\mathbb{R}^{d}}\Int_{s}^{T}\varphi(y)f_r(X_{t,r}(y),\overline{u}(r,X_{t,r}(y)),(\nabla \overline{u}\sigma)(r,X_{t,r}(y)))drdy\\
& +\Int_{\mathbb{R}^{d}}\Int_{s}^{T}\varphi(y)h_r(X_{t,r}(y),\overline{u}(r,X_{t,r}(y)),(\nabla \overline{u}\sigma)(r,X_{t,r}(y)))dyd\W_r \\
&+\Int_{s}^{T}\Int_{\mathbb{R}^{d}}\phi (y)1_{\{\overline{u}\in\partial D\}}(r,X_{t,s}(y))d\overline{K}_{r}^{t,y}dy-\Int_{s}^{T}\Int_{\mathbb{R}
^{d}}\varphi (y)(\nabla \overline{u}\sigma)(r,X_{t,r}(y))dydB_{r}
\end{align*}
Since $\varphi $ is arbitrary, we can prove that for $w(y)dy$ almost every $
y$, ($\overline{u}(s,X_{t,s}(y))$, $(\nabla \overline{u}\sigma
)(s,X_{t,s}(y))$, $\widehat{K}_{s}^{t,y}$) solves the RBDSDE 
$(\Phi(X_{t,T}(y)),f,h)$. Here $\widehat{K}_{s}^{t,y}$=$\Int_{t}^{s}1_{\{%
\overline{u}\in\partial D\}}(r,X_{t,r}(y))d\overline{K}_{r}^{t,y}$. Then by the
uniqueness of the solution of the RBDSDE, we know $\overline{u}%
(s,X_{t,s}(y))=Y_{s}^{t,y}=u(s,X_{t,s}(y))$, $(\nabla \overline{
u}\sigma)(s,X_{t,s}(y))=Z_{s}^{t,y}=(\nabla u\sigma)(s,X_{t,s}(y))$,  and $%
\widehat{K}_{s}^{t,y}=K_{s}^{t,y}$. Taking $s=t$ we deduce that $\overline{u}%
(t,y)=u(t,y)$, $w(y)dy$-a.s. and by the probabilistic interpretation (
\ref{con-k}), we obtain
$$
\Int_{s}^{T}\Int_{\R^d} \varphi _{t}(r,x)1_{\{\overline{u}\in\partial D\}}(r,x)\overline{\nu }
(dr,dx)=\Int_{s}^{T}\Int_{\R^d} \varphi _{t}(r,x)1_{\{u\in \partial D\}}(r,x)\nu (dr,dx).
$$
So $1_{\{\overline{u}\in\partial D\}}(r,x)\overline{\nu
}(dr,dx)=1_{\{u\in\partial D\}}(r,x)\nu (dr,dx)$.\\
\ep 

\appendix
\section{Some properties of convexity}
\label{Some properties of convexity}
In this section, we will list some properties of convex domains.
Denote $\rho(x)\triangleq (d(x,D))^2=|x-\pi(x)|^2$, the square of the distance to the domain $D$ and $\pi$ the projection on the closure $\bar{D}$. If $D$ is a convex domain, then the function $\rho$ is convex. Moreover, if $D$ is a regular domain, $\rho$ is two times differentiable on the complement of $D$ and we obtain:
\b*
\forall x\notin D, \quad \nabla\rho(x)=2(x-\pi(x)).
\e*
From this expression of gradient $\nabla\rho$, we remark that the hessian matrix $Hess \rho(x)$ has the following form:
$$Hess \rho(x) = 
\begin{pmatrix}
2 & 0& \cdots &0\\
0\\
\vdots & & (M)\\
0
\end{pmatrix}$$
where $M$ is a positive semidefinite matrix. We deduce also that:
\be \label{hess}
\forall z\in\R^{k\times d}\quad \text{trace}[zz^\ast Hess \rho(x)]\geq 0
\ee
Since $Hess$ is a positive matrix we have for every unit outward normal $n(x)$
\be
|z (n(x))^\ast|^2&\leq& C\Frac{1}{2}~\text{trace}\Big[zz^\ast\begin{pmatrix}
2 & 0& \cdots &0\\
0\\
\vdots & & (M)\\
0
\end{pmatrix}\Big]\nonumber\\
&\leq & C ~\text{trace}[zz^\ast Hess \rho(x)]\label{hessZ}
\ee
\begin{Lemma}\label{convprop1}
If $D_\varepsilon$ is a convex domain which satisfies (\ref{propappro}), then $\exists c>0$ such that $\forall \varepsilon <1, \forall x\in\R^k$
\b*
|\pi(x)-\pi_{\varepsilon}(x)| < c\sqrt{\varepsilon^2+\varepsilon d(x,D_\varepsilon)}\quad \mbox{and}\quad |\pi(x)-\pi_\varepsilon(x)| < c\sqrt{\varepsilon^2+\varepsilon d(x,D)}
\e*
where $\pi_\varepsilon(x)$ is the projection of $x\in\R^k$ on the closure $\bar{D}_\varepsilon$
\end{Lemma}
\begin{Lemma}\label{convprop2}
$$\exists c>0 \quad\mbox{such that}\quad \forall \varepsilon <1,~ \forall x\in\R^k~~
|\pi(x)-\pi_{\varepsilon}(x)| < c\sqrt{\varepsilon}(1+d(x,D_\varepsilon))$$
$$\exists c>0 \quad\mbox{such that}\quad \forall \varepsilon <1,~ \forall x\in\R^k~~
|\pi(x)-\pi_{\varepsilon}(x)|{\bf{1}}_{\{d(x,D_\varepsilon)>\varepsilon\}} < c\sqrt{\varepsilon}\sqrt{d(x,D_\varepsilon)}{\bf{1}}_{\{d(x,D_\varepsilon)>\varepsilon\}}$$
\end{Lemma}

\section{A priori estimates}
\label{A priori estimates}
In this section, we provide a priori estimates which are uniform in $n$ on the solutions of (\ref{BDSDEpen}).
\begin{Lemma}\label{estapriori}
There exists a constant $C>0$, independent of $n$, such that for all $n$ large enough
\be
\underset{n}{\Sup}\,\E\,\big[\underset{0\leq t\leq T}{\Sup}|Y_t^n|^2+\Int_t^T \|Z_s^n\|^2 ds + \|K^n\|_{VT}\big]\leq C\E\big[|\xi^2|+\Int_t^T \big(|f_s^0|^2+|h_s^0|^2\big)ds\big]
\label{estunif}. 
\ee
\end{Lemma}
\begin{proof}
For a given point $a\in D$, that satisfies condition (\ref{prop3}), we apply generalized It\^o's formula to get
\begin{align}\label{estimationuniforme}
|Y_t^n-a|^2+\Int_t^T &\|Z_s^n\|^2ds= |\xi-a|^2 + 2\Int_t^T (Y_s^n-a)^\ast f_s(Y_s^n,Z_s^n)ds+2\Int_t^T (Y_s^n-a)^\ast h_s(Y_s^n,Z_s^n)d\W_s\nonumber\\
 &-2\Int_t^T(Y_s^n-a)^\ast Z_s^n dB_s+\Int_t^T\|h_s(Y_s^n,Z_s^n)\|^2ds- 2n\Int_t^T (Y_s^n-a)^\ast(Y_s^n-\pi(Y_s^n))ds.
\end{align}
The stocastic integrals have both zero expectations since $(Y^n,Z^n)$ belongs to ${\mathcal S}^2_k([0,T])\times{\mathcal H}^2_{k\times d}([0,T])$. We take expectation in (\ref{estimationuniforme}) and we use conditions (\ref{prop1}), (\ref{prop3}) and the Lipschitz Assumption \ref{Ass2} in order to obtain
\begin{align}
\begin{split}
\E[|Y_t^n-a|^2]&+\E[\Int_t^T \|Z_s^n\|^2ds]\leq\E[|\xi-a|^2]+2C\E[\Int_t^T |Y_s^n-a|(|f_s(a,0)|+ |Y_s^n-a|+ \|Z_s^n\|)ds]\\
&+\E[\Int_t^T|h_s(a,0)|^2ds] + C\E[\Int_t^T |Y_s^n-a|^2ds]+ \alpha\E[\Int_t^T \|Z_s^n\|^2ds]\\
&\leq C(1+\E[\Int_t^T (|f_s(a,0)|^2+|h_s(a,0)|^2)ds] +C(1+\varepsilon^{-1})\E[\Int_t^T |Y_s^n-a|^2ds]\\
&+(\alpha+\varepsilon)\E[\Int_t^T \|Z_s^n\|^2ds].
\end{split}
\end{align}
Thus, if we choose $\varepsilon=\Frac{1-\alpha}{2}$, we have 
\begin{align*}
\E[|Y_t^n-a|^2]+(\Frac{1-\alpha}{2})\E[\Int_t^T \|Z_s^n\|^2ds]&\leq  C(1+\E[\Int_t^T |Y_s^n-a|^2ds]).
\end{align*}
Then, it follows from Gronwall's lemma that 
$$\underset{0\leq t\leq T}{\Sup}\E[|Y_t^n-a|^2]\leq Ce^{CT}.$$
Therefore we can deduce 
\be
\underset{0\leq t\leq T}{\Sup}\E[|Y_t^n|^2]\leq C\quad\text{and}\quad \E[\Int_0^T \|Z_s^n\|^2ds]\leq C.
\ee
On the other hand, the uniform estimate on $Y^n$ is obtained by taking the supremum over $t$ in the equation (\ref{estimationuniforme}), using the previous calculations and Burkholder-Davis-Gundy inequality. Thus, we get for all $n\geq 0$
\b*
\E[\underset{0\leq t\leq T}{\Sup}|Y_t^n-a|^2]\leq C \quad \text{and}\quad \E[\underset{0\leq t\leq T}{\Sup}|Y_t^n|^2]\leq C.
\e*
Finally, the total variation of the process $K^n$ is given by
\b*
\|K^n\|_{VT}= n\Int_0^T |Y_s^n-\pi(Y_s^n)|ds.
\e*
But from the property (\ref{prop3}) and the equation (\ref{estimationuniforme}) we have
\begin{align*}
\begin{split}
2n\Int_t^T \gamma |Y_s^n-\pi(Y_s^n)|ds &\leq  2n\Int_t^T |(Y_s^n-a)^\ast(Y_s^n-\pi(Y_s^n))|ds\\
&\leq  |\xi-a|^2 + 2\Int_t^T (Y_s^n-a)^\ast f_s(Y_s^n,Z_s^n)ds+2\Int_t^T (Y_s^n-a)^\ast h_s(Y_s^n,Z_s^n)d\W_s\nonumber\\
 &- 2\Int_t^T(Y_s^n-a)^\ast Z_s^n dB_s+\Int_t^T\|h_s(Y_s^n,Z_s^n)\|^2ds.
 \end{split}
\end{align*}
Hence it follows from previous estimates that 
\b*
\E[ \|K^n\|_{VT}]\leq C,
\e*
and the proof of Lemma \ref{estapriori} is complete.\ep
\end{proof}
\section{Proof of Lemma \ref{extraestimate}}
\label{Proof of Lemma}
Let first recall that $(Y^n,Z^n)$ is solution of the BDSDE \eqref{BDSDEpen} associated to $(\xi,f^n,h)$ where $f_s^n(y,z)=f_s(y,z)-n(y-\pi(y))$, for each $(y,z)\in\R^k\times\R^{k\times d}$. Note that , since we have assumed that $0\in D$  (Assumption \ref{assumptiondomain}), $f_{{\color{red}{s}}}^n(0,0)=f_s(0,0):=f_s^0$. Therefore, thanks to $L^p$-estimate for BDSDE (Theorem 4.1 in \cite{pp1994} applied for $p=4$), we have the following estimate
\be \label{estimateL4}
 \underset{n}{\Sup}\,\E[\underset{0\leq t\leq t}{\Sup}|Y_t^n|^4+\Big(\Int_0^T\|Z_s^n  \|^2ds\Big)^2] \leq C \E\Big[|\xi|^4+\Int_0^T (|f_s^0|^4+ |h_s^0|^4)ds \Big]<\infty.
\ee
Now, we apply generalized It\^o's formula to get
\begin{align}\label{estimat}
|Y_t^n|^2+\Int_t^T &\|Z_s^n\|^2ds+ 2n\Int_t^T (Y_s^n)^\ast(Y_s^n-\pi(Y_s^n))ds= |\xi|^2 + 2\Int_t^T (Y_s^n)^\ast f_s(Y_s^n,Z_s^n)ds\nonumber\\
&+2\Int_t^T (Y_s^n)^\ast h_s(Y_s^n,Z_s^n)d\W_s
 -2\Int_t^T(Y_s^n)^\ast Z_s^n dB_s+\Int_t^T\|h_s(Y_s^n,Z_s^n)\|^2ds.
\end{align}
Without loss of generality we apply the property (\ref{prop3}) for $a=0$ since $0\in D$ (If not, we apply the generalized It\^o formula to $|Y^n_t-a|^2$) to obtain
\begin{align*}
\begin{split}
2n\Int_0^T \gamma |Y_s^n-\pi(Y_s^n)|ds &\leq  2n\Int_0^T |(Y_s^n)^\ast(Y_s^n-\pi(Y_s^n))|ds\\
&\leq  |\xi|^2 + 2\Int_0^T (Y_s^n)^\ast f_s(Y_s^n,Z_s^n)ds+2\Int_0^T (Y_s^n)^\ast h_s(Y_s^n,Z_s^n)d\W_s\nonumber\\
 &- 2\Int_0^T(Y_s^n)^\ast Z_s^n dB_s+\Int_0^T\|h_s(Y_s^n,Z_s^n)\|^2ds.
 \end{split}
\end{align*}
Then, taking the square and the expectation yields
\begin{align*}
\begin{split}
\E\Big[\Big(n\Int_0^T |Y_s^n-\pi(Y_s^n)|ds\Big)^2\Big] 
&\leq C \E\Big[|\xi|^4\Big] + C\E\Big[\Big(\Int_0^T (Y_s^n)^\ast f_s(Y_s^n,Z_s^n)ds\Big)^2\Big]\nonumber\\ 
&+ C\E\Big[\Big(\Int_0^T (Y_s^n)^\ast h_s(Y_s^n,Z_s^n)d\W_s\Big)^2\Big]
+ C\E\Big[\Big(\Int_0^T(Y_s^n)^\ast Z_s^n dB_s\Big)^2\Big]\nonumber\\
&+C\E\Big[\Big(\Int_0^T\|h_s(Y_s^n,Z_s^n)\|^2ds\Big)^2\Big].
 \end{split}
\end{align*}
By using the isometry property and the boundedness of $f$ and $h$, we obtain
\begin{align*}
\begin{split}
\E\Big[\Big(n\Int_0^T |Y_s^n-\pi(Y_s^n)|&ds\Big)^2\Big] 
\leq C \E\Big[|\xi|^4\Big] + C\E\Big[\Big(\Int_0^T {\color{red}{|}}Y_s^n{\color{red}{|}} ds\Big)^2\Big]\nonumber\\ 
&+ C\E\Big[\Int_0^T |Y_s^n h_s(Y_s^n,Z_s^n)|^2ds\Big]
+C\E\Big[\Int_0^T|Y_s^n Z_s^n|^2ds\Big]+C.
 \end{split}
\end{align*}
Finally, we deduce from Holder inequality and boundedness of $h$ that 
\begin{align*}
\begin{split}
\E\Big[\Big(n\Int_0^T |Y_s^n-\pi(Y_s^n)|&ds\Big)^2\Big] 
\leq C \E\Big[|\xi|^4 +\Int_0^T |Y_s^n|^2 ds+\underset{0\leq t\leq T}{\Sup}|Y_t^n|^4+\Big(\Int_0^T\|Z_s^n  \|^2ds\Big)^2\Big]+C.
 \end{split}
\end{align*}
Thus, from  the estimate \eqref{estimateL4} we get the desired result.
\ep
\section{Proof of the tightness of the sequence $(\pi_n)_{ n\in \mathbb N}$}
\label{appendix:tight}
Recall first that  $\nu _{n}(dt,dx)=-n(u_{n}-\pi(u_{n}))(t,x)dtdx$ and $\pi _{n}(dt,dx)=w(x)\nu _{n}(dt,dx)$ where $ u_n $ is the solution of the SPDEs \eqref{o-equa1}.  

\begin{Lemma}
\label{tight}
 $\P$-a.s. in $\omega\in\Omega$, the sequence of random measure $(\pi _{n}(\omega,ds,dx))_{n  \in \mathbb N}$  is tight.
\end{Lemma}
\begin{proof}  
We shall prove that for every $\epsilon>0$ , there exists some constant $K$ such that
\be
\E\big[\Int_0^T\Int_{\mathbb R^d} \textbf{1}_{\left\lbrace \lvert x\rvert\geq 2K\right\rbrace} |\pi_n|(ds,dx)\big]\leq \epsilon,\ \forall n\in \N.
\ee
We first write
\begin{align*}
&\Int_0^T\Int_{\mathbb R^d} \textbf{1}_{\left\lbrace \lvert x\rvert\geq 2K\right\rbrace} |\pi_n|(ds,dx)\\&= \Int_0^T\Int_{\mathbb R^d} \textbf{1}_{\left\lbrace \lvert x\rvert\geq 2K\right\rbrace} \left( \textbf{1}_{\left\lbrace \left|X^{-1}_{0,s}(x)-x\right|\leq K\right\rbrace} + \textbf{1}_{\left\lbrace \left|X^{-1}_{0,s}(x)-x\right|\geq K\right\rbrace} \right)| \pi_n|(ds,dx)\\
&:=I^n_K+L^n_K, \quad \P-a.s.
\end{align*}
Taking expectation yields 
\be 
\E\big[\Int_0^T\Int_{\mathbb R^d}
\textbf{1}_{\left\lbrace \lvert x\rvert\geq 2K\right\rbrace}
\pi_n(ds,dx)\big]=\E \big[I^n_K\big]+\E\big[L^n_K\big]. \ee
From Definition $\ref{o-pde} (ii) a)$, we have $\pi _{n}(\omega,ds,dx)$ is a $\Fc_{s,T}^W$-adapted measure and we know that the inverse of the flow depends only on the noise $B$. Then, since $W$ and $B$ are mutually independant we get by (\ref{est-measure}) and for $K \geq 2\lVert b \rVert_{\infty}T $, we get 
\b*
\E \big[L^n_K\big] &\leq &  \P\left(\underset{0\leq r\leq T}{{\rm sup}}\left|X^{-1}_{0,r}(x)-x\right|\geq K\right) \E\big[\Int_0^T\Int_{\mathbb R^d}|\pi_n|(ds,dx)\big]\\
&\leq& \left(  C_1\ {\rm exp}(-C_2K^2)+C_3\ {\rm exp}(-C_4K) \right) \E\big[|\pi_n|([0,T]\times \mathbb R^d)\big]\\
&\leq& C'_1\ {\rm exp}(-C_2K^2)+C'_3\ {\rm exp}(-C_4K). 
\e*
 Finally, for $K$ sufficiently large we obtain
  $$\E[ L^n_K] \leq\epsilon .$$ 
On the other hand, if $\left|x\right|\geq 2K $ and $\left|X^{-1}_{0,s}(x)-x\right|\leq K $ then $\left|X^{-1}_{0,s}(x)\right|\geq K $. Therefore
\b*
\E[ I^n_K] &\leq & \E [\Int_0^T\Int_{\mathbb R^d}\textbf{1}_{\left\{\left|X^{-1}_{0,s}(x)\right|\geq K \right\}}w(x)|\nu_n|(ds,dx)]\\
&=&\E[ \Int_0^T\Int_{\mathbb
R^d}\textbf{1}_{\left\{\left|X^{-1}_{0,s}(x)\right|\geq K
\right\}}w(x)n|u_{n}-\pi(u_{n})|(s,x)dsdx] \e*
which, by the change
of variable $y= X^{-1}_{0,s}(x)$, becomes
\begin{align*}
& \E\left[\Int_0^T\Int_{\mathbb R^d}\textbf{1}_{\left\{\left| y\right|\geq K \right\}}w(X_{0,s}(y))J(X_{0,s}(y))n|u_{n}-\pi(u_{n})|(s,X_{0,s}(y))dsdy\right] \\ 
&\leq \E\left[\int_{\mathbb R^d}w(x)\left(w(x)^{-1}\textbf{1}_{\left\{\left|x\right|\geq K\right\}} \underset{0\leq r\leq T}{{\rm sup}}w(X_{0,r}(x))J(X_{0,r}(x))\right)\|K^{n,0,x}\|_{VT}dx\right] \\
&\leq \left(\E\left[\Int_{\mathbb R^d} \|K^{n,0,x}\|_{VT}^2 w(x)dx\right]\right)^{1/2}\\ 
&\hspace{0.5cm}\left(\E \left[\Int_{\mathbb R^d} \left(w(x)^{-1} \textbf{1}_{\left\{\left|x\right|\geq K\right\}}\underset{0\leq r\leq T}{{\rm sup}}w(X_{0,r}(x))J(X_{0,r}(x))\right)^2 w(x)dx\right]\right)^{1/2} \\
&\leq C \left(\E\left[ \Int_{\mathbb R^d} \left(w(x)^{-1}
\textbf{1}_{\left\{\left|x\right|\geq K\right\}}\underset{0\leq r\leq T}{{\rm sup}} w(X_{0,r}(x))J(X_{0,r}(x))\right)^2 w(x)dx\right]\right)^{1/2}.
\end{align*}
where the last inequality is a consequence of (\ref{estunif}). It is now sufficient to prove that 
\be \Int_{\mathbb
R^d} w(x)^{-1}\E\left[\left(\underset{0\leq r\leq T}{{\rm
sup}}w(X_{0,r}(x))J(X_{0,r}(x))\right)^2\right]dx<\infty . \ee
Note that
\begin{align*}
&\E\left[\left(\underset{0\leq r\leq T}{{\rm sup}} w(X_{0,r}(x))J(X_{0,r}(x))\right)^2\right]\\ &\leq \left[\E\left(\underset{0\leq r\leq T}{{\rm sup}}\left|w(X_{0,r}(x))\right|\right)^4\right]^{1/2}\left[\E\left(\underset{0\leq r\leq T}{{\rm sup}}\left|J(X_{0,r}(x))\right|\right)^4\right]^{1/2}\\
&\leq C\left[\E\left(\underset{0\leq r\leq T}{{\rm
sup}}\left|w(X_{0,r}(x))\right|\right)^4\right]^{1/2}.
\end{align*}
Therefore it is sufficient to prove that:
\begin{equation*}
\Int_{\mathbb R^d}\frac{1}{w(x)}\left(\E\left[\underset{t\leq
r\leq T}{{\rm
sup}}\left|w(X_{t,r}(x))\right|^4\right]\right)^{1/2}dx<\infty.
\end{equation*}
Since $w(x)\leq 1$, we have
\b*
\E\left[\underset{t\leq r\leq T}{{\rm sup}}\left|w(X_{t,r}(x))\right|^4\right]&\leq & \E\left[\underset{t\leq r\leq T}{{\rm sup}}\left|w(X_{t,r}(x))\right|^4\textbf{1}_{\left\lbrace \underset{t\leq r\leq T}{{\rm sup}}\left|X_{t,r}(x)-x\right|\leq\frac{\left|x\right|}{2}\right\rbrace }\right]\\
& &+\P\left(\underset{t\leq r\leq T}{{\rm sup}}\left|X_{t,r}(x)-x\right|\geq\frac{\left|x\right|}{2}\right)\\
&=&:A(x)+B(x).
\e*
If $\underset{t\leq r\leq T}{{\rm sup}}\left|X_{t,r}(x)-x\right|\leq\frac{\left|x\right|}{2}$ then $\left|X_{t,r}(x)\right|\geq\frac{\left|x\right|}{2}$ and so $\left|w(X_{t,r}(x))\right|\leq \left(1+\frac{\lvert x \rvert}{2} \right)^{-p}$. Thus we have that $A(x)\leq \left(1+\frac{\lvert x \rvert}{2} \right)^{-4p}$ and so $\int_{\mathbb R^d} \left(1+\lvert x \rvert \right)^{p}A(x)^{1/2}dx<\infty$. On the other hand, if $\lvert x \rvert \geq 4\lVert b \rVert_{\infty}T$, then (the same argument as in the existence proof step 2 of Theorem 4 in \cite{BCEF} for the It\^o integral with respect to the Brownian motion) 
\begin{align*}
\label{large_deviation}
B(x) &\leq \P\left(\underset{t\leq s\leq T}{{\rm sup}}\left|\int_0^s\sigma(X_{0,r}(x))dW_r\right|\geq \frac{\lvert x \rvert}{8} \right)\\
&\leq C_1\ {\rm exp}(-C_2\lvert x \rvert^2)
\end{align*}
and so $\int_{\mathbb R^d} \left(1+\lvert x \rvert \right)^{p}B(x)^{1/2}dx<\infty$.

\end{proof}
\bibliographystyle{acm}
\bibliography{Thesis}

%
%

\end{document}